\documentclass[a4paper,11pt]{article}

\usepackage{color}
\usepackage{amsmath}
\usepackage{amssymb}
\usepackage{amsfonts}
\usepackage{amsthm}
\usepackage{mathrsfs} 
\usepackage{graphicx}
\usepackage{psfrag}
\usepackage{verbatim}
\usepackage{paralist}
\usepackage{bm}
\newcommand{\bbE}{\mathbb{E}}
\newcommand{\bbV}{\mathbb{V}}
\newcommand{\bbX}{\mathbb{X}}
\newcommand{\Var}{\bbV{\rm ar}}
\newcommand{\bbP}{\mathbb{P}}

\newcommand{\bbZ}{\mathbb{Z}}
\newcommand{\bbR}{\mathbb{R}}

\newcommand{\ol}{\bar}

\newcommand{\wh}{\widehat}

\newcommand{\cA}{\mathcal A}

\newcommand{\cZ}{\mathcal Z}

\newcommand{\cR}{\mathcal R}

\newcommand{\rmc}{{\rm c}}
\newcommand{\rmC}{{\rm C}}

\newcommand{\rmd}{{\rm d}}
\newcommand{\rme}{{\rm e}}

\newcommand{\supp}{{\rm supp}}
\newcommand{\sgn}{{\rm sgn}}
\newcommand{\wt}{\widetilde}

\newcommand{\la}{\langle}
\newcommand{\ra}{\rangle}

\newtheorem{theorem}{Theorem}[section]

\newtheorem{lem}[theorem]{Lemma}
\newtheorem{prop}[theorem]{Proposition}
\newtheorem{thm}[theorem]{Theorem}

\numberwithin{equation}{section}

\title{Large Deviations for the Empirical Distribution in the
	General Branching Random Walk}
\author
{Oren Louidor\thanks{oren.louidor@gmail.com, eliadtsai@gmail.com} \\Technion, Israel \and Eliad Tsairi\footnotemark[1]\\Technion, Israel}
\date{}

\begin{document}
\maketitle

\begin{abstract}
We consider the general branching random walk under minimal assumptions, which in particular guarantee that the empirical particle distribution admits an almost sure central limit theorem. For such a process, we study the large time decay of the probability that the fraction of particles in a typical set deviates from its typical value. We show that such probabilities decay doubly exponentially with speed which is either linear in or the square root of the number of generations, depending on the set and the magnitude of the deviation. We also find the rate of decay in each case.
\end{abstract}

\section{Introduction and Results}
\subsection{Setup}
\label{ss:1}
In this manuscript we study probabilities of large deviations in the empirical distribution of the branching random walk, under very general assumptions. We denote this process by the sequence $Z=(Z_n)_{n \geq 0}$ of random point measures on $\bbR$. The law of $Z$ is completely determined by the distribution of $Z_1$, via the recursive relation
\begin{equation}
\label{e:1.1}
Z_0 = \delta_0 
\quad; \qquad
Z_n  \,\big|\, Z_0, \dots, Z_{n-1} \, \overset{\rmd}= \, \sum_{x \in Z_{n-1}} Z_1^{(x)}(\cdot - x) \ ,\ \  n \geq 2 \,,
\end{equation} 
where $\{Z_1^{(x)} :\: x \in Z_{n-1}\}$ are independent and distributed as $Z_1$. Here and after, given a finite point measure $\zeta$, we shall write $(x :\: x \in \zeta)$ for the sequence of atoms of $\zeta$, in an arbitrary order, with each atom $x$ repeating $\zeta(\{x\})$ many times. Moreover, whenever $(x :\: x \in \zeta)$ is used as an index set (e.g. in a summation), each such repetition is thought of as a different index. 

The usual interpretation of the measure $Z_n$ is that of representing a population of particles, whose locations in generation $n$ are given by $(x:\: x \in Z_n)$. Then~\eqref{e:1.1} models an evolution of such population, in which starting from a single particle at the origin, a new generation of particles is formed from the previous generation by having all particles give birth to a random population of offspring and die, with each such offspring population chosen independently and, relative to the location of the parent, according to the same generation-invariant law. We shall refer to this law, which by definition is also the distribution of $Z_1$, as the {\em offspring distribution} or {\em reproduction law}.

The following assumptions on the reproduction law will be imposed: (Below and from this point on, given a Borel measure $\zeta$ on $\bbR$, we use the notation $|\zeta|$ to denote $\zeta(\bbR)$ and for $f: \bbR \to \bbR$, write $\la \zeta, f \ra$ for the integral of $f$ with respect to the measure $\zeta$.)

\noindent {\bf (A1)} $\bbE |Z_1| \log |Z_1| < \infty$.
\\ \noindent {\bf (A2)} $\bbE \la Z_1, x \ra = 0$ and $\bbE \la Z_1, x^2 \ra / \bbE |Z_1|  = 1$.
\\ \noindent {\bf (A3)} $\log 2 \leq b_0 := \text{ess-inf} (\log |Z_1|) \leq \text{ess-sup} (\log |Z_1|) =: b_1 \leq \infty$.
\\ \noindent {\bf (A4)} There exists $r_0 < \infty$ and a discrete subset $\cR \subseteq [-r_0, r_0]$, 
such that $\supp(Z_1) \subseteq \cR$ a.s.
\\ \noindent {\bf (A5)} $\Var \big(|Z_1|\big) + \Var \big(\la Z_1, x \ra \big) > 0$.

\noindent These assumptions are rather general and, with a few exceptions (made only for convenience), are minimal for our main result to hold. See Subsection~\ref{sss:Assumptions} for a further discussion.

Letting $\ol{Z}_n := Z_n / |Z_n|$, it is well known that under (A1) - (A3), almost-surely as $n \to \infty$,
\begin{equation}
\label{e:1}
|Z_n|/(\bbE |Z_1|)^n  \longrightarrow W
\quad, \qquad
\ol{Z}_n(\sqrt{n} \cdot) \Longrightarrow \nu \,,
\end{equation}
where $W$ is a positive random variable with mean $1$ and $\nu$ is the standard Gaussian measure on $\bbR$. In particular, letting $\cA$ be the algebra generated by all subsets of the form $(-\infty, x]$ for $x \in \bbR$, the latter implies that for all $A \in \cA$, 
\begin{equation}
\label{e:1.3}
\lim_{n \to \infty} \ol{Z}_n(\sqrt{n}A) = \nu(A) \quad \text{a.s.} 
\end{equation}
The convergence of $|Z_n|/\bbE|Z_n|$ to a mean $1$ random variable, assuming only (A1) and (A3), was initially shown by Levinson~\cite{levinson1959limiting}. The necessity of Assumption (A1) was shown later by Kesten and Stigum~\cite{KestenStigum66}. The almost-sure weak convergence of $\ol{Z}_n(\sqrt{n} \cdot)$ to $\nu$, assuming only (A1)-(A3), was proved by Biggins in~\cite{biggins1990central}. The latter extended an earlier work by Kaplan~\cite{Kaplan82}, who assumed in addition that the displacements of born particles are independent of each other and of their number (sometimes referred to as the {\em mixed-sample process}).

The goal of the present work is therefore to study asymptotics of probabilities of the form
\begin{equation}
\bbP \big(\ol{Z}_n(\sqrt{n}A) \geq p \big) \,,
\end{equation}
as $n \to \infty$, for $A \in \cA$ and $p \in (\nu(A), 1)$. Although this is a natural problem to consider, as far as we know, the only works in this direction are that of Louidor and Perkins ~\cite{louidor2015large}, where such probabilities were studied, under the additional assumption of independence between the displacements of born particles and their number, and that of Chen and He~\cite{chen2017large}, where such independence is again assumed but assumptions (A3) and (A4) are relaxed (see~\cite{louidor2015large} for a discussion of different large deviation results concerning the branching random walk). 

Our aim here is therefore to extend these results to the general case. We shall show that such probabilities decay double exponentially fast with speed $\sqrt{n}$ or $n$, depending on $A$ and $p$, at rate $I(A,p)$ or $J(A,p)$ respectively. In order to state the full theorem, let us first define the rate functions $I(A,p)$ and $J(A,p)$, corresponding to each regime of decay. To this end, if $n \geq 0$ and $b \in [b_0, b_1] \cap \bbR$ we set:
\begin{equation}
\cZ_n := \big\{\zeta :\: \bbP(Z_n = \zeta) > 0 \big\}
\quad, \qquad
\cZ_n(b) := \big \{\zeta \in \cZ_n :\: |\zeta| = \rme^{bn} \} 
\quad, \qquad
\cZ := \cup_{n \geq 0} \cZ_n \,.
\end{equation}
These are the set of possible particle measures for the process in generation $n$, in generation $n$ and with mean log-reproduction rate $b$, or without any restrictions, respectively. We note that under our assumptions $\cZ_n$, $\cZ$ are always countable, while for some values of $b$ and $n$ we may have $\cZ_n(b) = \emptyset$.  We then let
\begin{equation}
\label{e:7}
\alpha^\pm_n := \sup \big \{\la\ol{\zeta},\,  \pm x \ra
:\: \zeta \in \cZ_n \big\} 
\quad, \qquad
\alpha^\pm := \sup_{n \geq 1} \alpha^\pm_n/n \,.
\end{equation}
That is, $\alpha^\pm$ is the maximal rate of increase/decrease in the average particle height.
It is clear from Assumption (A4) that $|\alpha^\pm| \leq r_0$. Positivity of these quantities as well as the monotonicity of $n \mapsto \alpha^\pm_n$ is given by:
\begin{prop}
\label{p:1.2b}
Under Assumptions (A1) - (A5) we have $\alpha^\pm_n \uparrow \alpha^\pm \in (0, r_0]$ as $n \to \infty$.
\end{prop}

As in~\cite{louidor2015large}, we next define the functions 
$\wt{I}^{\pm} : \cA \times [0,1] \to [0, \infty]$ and $\wt{J}: \cA \times [0,1] \to [0,1]$ by,
\begin{equation}
\label{e:1.5}
\wt{I}^\pm(A,p) := \inf \big\{ x \in \bbR_+ :\: \nu(A \mp x) \geq p \big\} \,,
\end{equation}
and
\begin{equation}
\label{e:1.6}
\wt{J}(A,p) := \inf \left\{ u \in [0,1) :\: p \leq \sup_{x \in \bbR} \nu\Big(\tfrac{A}{\sqrt{1-u}} - x\Big) \right\} \,.
\end{equation}
Then the rate functions $I, J: \cA \times [0,1] \to [0, \infty]$ are given by
\begin{equation}
\label{e:1.11a}
I(A,p) := b_0 \inf \big\{  \wt{I}^+(A,p)/\alpha^+ \,,\,\, \wt{I}^-(A,p)/\alpha^- \big\}
\quad, \qquad 
J(A,p) := b_0 \wt{J}(A,p) \,.
\end{equation}

\medskip
We can now state the main theorems of this manuscript.
\begin{thm}
\label{thm:1}
Let $A \in \cA \setminus \{\emptyset\}$ and $p \in (\nu(A),1)$. If $I(A,p) < \infty$ then 
\begin{equation}
\label{e:5}
\lim_{n \to \infty} \frac{1}{\sqrt{n}} \log \Big[-\log 
	\bbP \big(\ol{Z}_n(\sqrt{n}A) \geq p \big)\Big] = I(A,p) \in (0, \infty) \,,
\end{equation}
whenever $I(A, \cdot)$ is continuous at $p$. If $I(A,p) = \infty$ then
\begin{equation}
\label{e:6}
\lim_{n \to \infty} \frac{1}{n} \log \Big[-\log 
	\bbP \big(\ol{Z}_n(\sqrt{n}A) \geq p \big)\Big] = J(A,p) \in (0,\infty) \,,
\end{equation}	
whenever $J(A, \cdot)$ is continuous at $p$.
\end{thm}
\noindent
Replacing $A$ by $A^c$ and $p$ by $1-p$ in Theorem~\ref{thm:1}, we immediately get,
\begin{thm}
\label{thm:1'}
Let $A \in \cA \setminus \{\bbR\}$ and $p \in (0, \nu(A))$. If $I(A^{\rmc},(1-p)) < \infty$ then 
\begin{equation}
\lim_{n \to \infty} \frac{1}{\sqrt{n}} \log \Big[-\log 
	\bbP \big(\ol{Z}_n(\sqrt{n}A) \leq p \big)\Big] = I(A^\rmc,1-p) \in (0, \infty) \,,
\end{equation}
whenever $I(A^\rmc, \cdot)$ is continuous at $1-p$. If $I(A^\rmc,1-p) = \infty$ then
\begin{equation}
\lim_{n \to \infty} \frac{1}{n} \log \Big[-\log 
	\bbP \big(\ol{Z}_n(\sqrt{n}A) \leq p \big)\Big] = J(A^\rmc,1-p) \in (0,\infty) \,,
\end{equation}	
whenever $J(A^\rmc, \cdot)$ is continuous at $1-p$.
\end{thm}

\subsubsection{Remarks}  
 The restriction to continuity points of the rate function is commensurate with the usual large deviation formulation applied to sets of the form $[p, \infty)$. Indeed, in an LDP formulation the upper bound on the $\limsup$ for the decay rate of the measure of $[p, \infty)$ is given as the infimum of the rate function over $(p, \infty)$, while the lower bound on the $\liminf$ is the infimum over $[p,\infty)$. Only at continuity points of the rate function is it guaranteed that these infima match and a limit for the rate of decay exists. In~\cite{louidor2015large} (Proposition~2) it was shown that there are indeed choices of sets $A$ and discontinuity points $p$ such that the limits in Theorem~\ref{thm:1} (and Theorem~\ref{thm:1'}) do not exist.

The restriction to intervals of the form $(-\infty ,x]$, $(y, x]$ and
$(y, \infty)$ in $\cA$ is quite arbitrary and the theorem still holds if $\cA$
is the algebra generated by sets of the form $(-\infty, x)$ or more generally,
the set of all finite unions of disjoint finite or infinite intervals which either contain their endpoints or do not, or contain only one of them. On the other hand, Theorem~\ref{thm:1} cannot be expected to hold for all Borel sets, nor even all continuity sets of $\nu$. Indeed, in Proposition~3 of~\cite{louidor2015large}, it is shown that for any $\alpha \in [1/2, 1]$ one can find a Borel set $A \subset \bbR$ and $p \in (0,1)$ such that the probability of $\{\ol{Z}_n(\sqrt{n}A) \geq p\}$ decays doubly exponentially with speed $n^{\alpha}$.

\subsection{Three Examples}
\label{ss:examples}
We now present three examples of offspring distributions, for which we can compute $\alpha^\pm$. The value of the rate functions for any $A$ and $p$ can then be easily deduced via~\eqref{e:1.5},~\eqref{e:1.6} and~\eqref{e:1.11a}. Aside from being the first natural examples to consider, these examples will also help us illustrate some of the key points in the discussion that follows.

\subsubsection{Independent Reproduction and Motion}
\label{sss:1} 
We start with the mixed-sample case, namely when offspring displacements are independent of each other and of the number of born particles. This case was treated in~\cite{louidor2015large}. Formally, we let $Z_1 := \sum_{k=1}^N \delta_{X_k}$, where $N \geq 2$ is an integer-valued random variable, and $X_1, X_2, \dots$ are i.i.d discrete and bounded random variables, which are also independent of $N$. We further suppose that $\bbE (N \log N) < \infty$, $\bbE X_1 = 0$ and $\bbE X_1^2 = 1$. It is easy to see that under these conditions, Assumptions (A1)-(A5) hold with some $b_1 \geq b_0 \geq \log 2$ and $r_0 \in (0,\infty)$.
 
Since $\zeta_1 := \rme^{b_0} \delta_{r^+} \in \cZ_1$, where $r^+ := \text{ess-sup} (X_1)$, we must have $\alpha^+ \geq \la \ol{\zeta}_1 \,,\, x \ra = r^+$. Since the opposite bound is trivial, it follows that $\alpha^+ = r^+$. A symmetric argument gives $\alpha^- = r^-$, where $r^- = - \text{ess-inf} (X_1)$ and Theorem~\ref{thm:1} effectively identifies with Theorem 1 of~\cite{louidor2015large} (where $r^\pm=1$ was assumed).

\subsubsection{Correlated Reproduction and Motion}
\label{sss:2}
Next, we consider the following simple example where the above independence is no longer assumed . Explicitly, let
\begin{equation}
\label{e:1.13}
Z_1 :=  \begin{cases}
	\rme^{b_0} \delta_{-r^-} & \quad \text{w.p. } 1-p \,, \\
	\rme^{b_1} \delta_{r^+} & \quad \text{w.p. } p \,,
\end{cases}
\end{equation}
where  $p \in (0,1)$, $r^\pm > 0$ and $\log 2 \leq b_0 < b_1 < \infty$ satisfy $\rme^{b_0}, \rme^{b_1} \in \bbZ$ and 
\begin{equation}
p \rme^{b_1} r^+ - (1-p) \rme^{b_0} r^- = 0
\quad, \qquad
\frac {p\rme^{b_1} (r^+)^2 + (1-p) \rme^{b_0} (r^-)^2}{p \rme^{b_1} + (1-p) \rme^{b_0}} = 1\,.
\end{equation}
The last two conditions ensure that Assumption (A2) holds.

Here again it is easy to see that $\alpha^\pm = r^\pm$. Clearly $\alpha^\pm \leq r^\pm$. On the other hand, for $\zeta := \rme^{b_1} \delta_{r^+} \in \cZ_1$ and $\zeta' := \rme^{b_0} \delta_{-r^-} \in \cZ_1$, we have $\la \ol{\zeta} \,,\, x \ra = r^+$ and $\la \ol{\zeta}' \,,\, -x \ra = r^-$, which provide the matching lower bounds

\subsubsection{Random Reproduction and Fixed Motion}
\label{sss:3}
The triviality of $\alpha^\pm$ and their immediate derivation in the first two examples, should not be seen as an indication of the usual case. Indeed, as the next example shows, even under a rather simple reproduction law, $\alpha^\pm$ are already nontrivial and finding them explicitly is not a simple matter. For this example, we let the number of offspring be random, but insist that half of the born particles always take a step up, while the rest take a step down. For simplicity we suppose that the magnitude of the steps is $1$ and that the reproduction rate assumes only two values. This formalizes as $Z_1 := N \big(\frac12 \delta_{-1} + \frac12 \delta_{+1}\big)$, where $N$ is a random variable taking the value $\rme^{b_0}$ with probability $p$ and $\rme^{b_1}$ with probability $1-p$, where $p \in (0,1)$ and $\rme^{b_0}, \rme^{b_1} \in 2\bbZ$ satisfy $2 \leq \rme^{b_0} < \rme^{b_1} < \infty$.

The values of $\alpha^\pm$ can be easily deduced from Proposition~\ref{p:1.2} below, which shows that the maximal increase in the average particle height, when the mean reproduction rate is fixed to be $b \in [b_0 \,,\, b_1]$ is given by $\gamma(b)$, where  
\begin{equation}
\label{e:1.16}
\gamma(b) := \begin{cases}
	0 & b = b_0 \,,\\
	\beta^{-1}\big((b_1 - b) \wedge \log 2\big) 
	& b \in (b_0, b_1] 
\end{cases}
\end{equation}
and $\beta:[0,1] \to [0, \log 2]$ is the function
\begin{equation}
\label{e:1.17}
\beta(x) := \tfrac12 \big((1+x) \log (1+x) + (1-x) \log(1-x)\big) \,.
\end{equation}
Observe that $\beta$ is the rate function in the large deviation principle for the sum of i.i.d. symmetric $\pm 1$ random variables, as given by Cram\'{e}r's theorem (see, e.g.~\cite{dembo2010large}, Chapter 2). 
\begin{prop}
\label{p:1.2}
For any $b \in [b_0, b_1]$, $n \geq 1$ and $\zeta \in \cZ_n(b)$ we have
\begin{equation}
\label{e:1.18}
n^{-1} \la \ol{\zeta} \,,\, x \ra \leq \gamma(b) \,.
\end{equation}
On the other hand, for any $\epsilon >0$ and $b \in [b_0, b_1]$, there is $n \geq 1$ and $\zeta \in \cZ_n$ such that
\begin{equation}
\label{e:1.19}
n^{-1} \la \ol{\zeta} \,,\, x \ra \geq \gamma(b) - \epsilon 
\quad, \qquad
n^{-1} \log |\zeta| \in [b-\epsilon, b+\epsilon] \,.
\end{equation}
\end{prop} 

As Figure~\ref{f:1} demonstrates, the behavior of $\gamma$ is qualitatively different, depending on whether $b_1 \geq b_0 + \log 2$ or not. In the former case, $\gamma(b) \equiv 1$ for all $b \in (b_0, b_1 - \log 2]$, while in the latter case $\gamma(b) \leq \beta^{-1} (b_1 - b_0) < 1$ for all $b \in [b_0, b_1]$. In both cases $\gamma$ has a jump discontinuity at $b = b_0$, but otherwise continuous on $(b_0, b_1]$. 
\begin{figure}[ht!]
	\centering
    \includegraphics[width=1\textwidth]{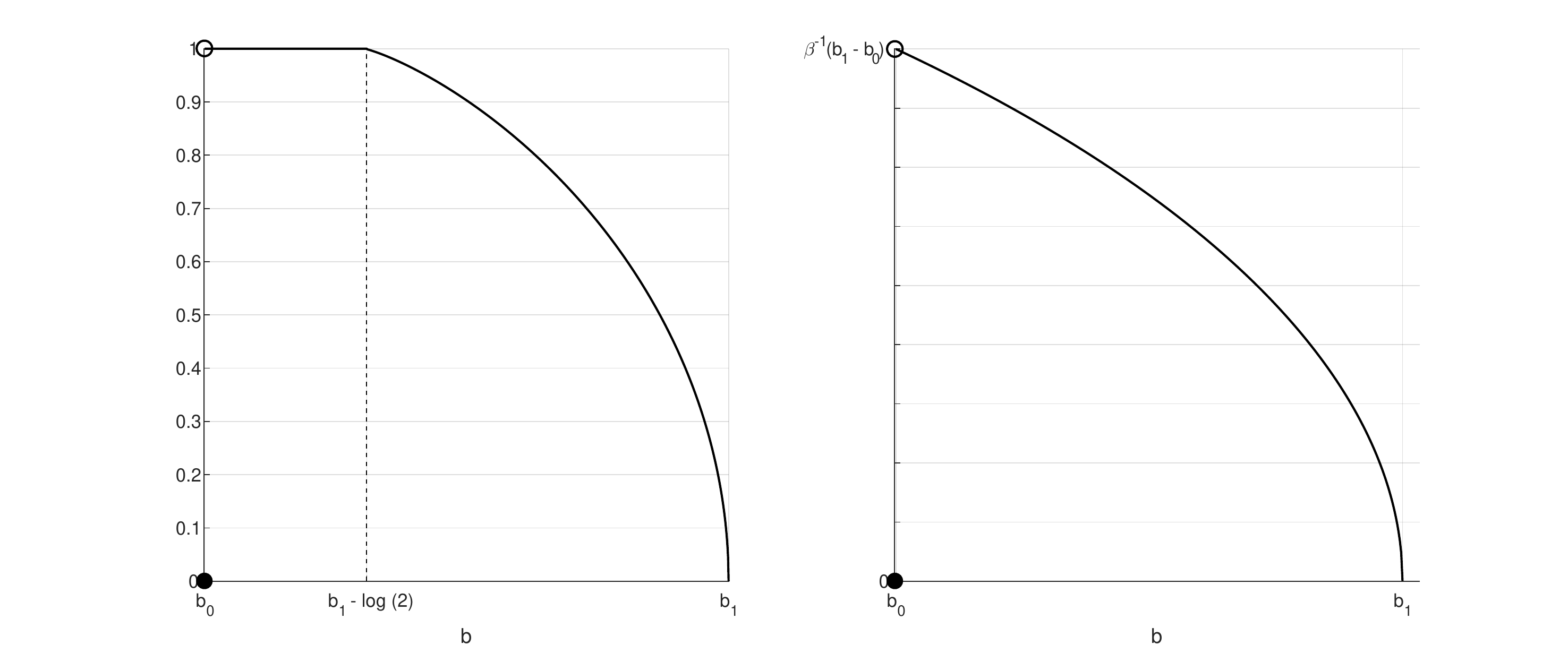} 
    \caption{The graph of $b \mapsto \gamma(b)$ in the case $b_1 \geq b_0 + \log 2$ (left) and $b_1 < b_0 +\log 2$ (right). In both cases $\gamma$ is discontinuous at $b_0$ and continuous at $b_1$.} 
    \label{f:1}
\end{figure}

It then follows immediately from the proposition and symmetry of the reproduction law that 
\begin{equation}
\alpha^\pm = \sup_{b \in [b_0, b_1]} \gamma(b) = \gamma(b_0+) = 
	\begin{cases}
		\beta^{-1}(b_1 - b_0) & b_1 < b_0 + \log 2  \,, \\
		1 & b_1 \geq b_0 + \log 2 \,.
	\end{cases}
\end{equation}

\subsection{Discussion}

\subsubsection{Optimal Strategy for Increase in the Empirical Fraction}
\label{sss:1.3.1}
As mentioned, this paper can be seen as a generalization of the work by Louidor and Perkins~\cite{louidor2015large}, in which Theorems~\ref{thm:1} and~\ref{thm:1'} were proved only for the case of Example~\ref{sss:1}, namely when motion and branching are independent, and moreover when the steps are symmetric $\pm 1$. For this, henceforth ``simple case'', the authors identified two ``strategies'' for realizing the event $\{\ol{Z}_n(\sqrt{n}A) \geq p\}$ for $p \in (\nu(A), 1)$ and showed that the ``probabilistic cost'' of these is the minimum possible.

The first is a {\em shift strategy}, whereby all particles move to $x\sqrt{n}$ in $|x|\sqrt{n}$ generations, by taking either the $+1$ or the $-1$ step consistently and reproducing at the minimal possible rate of $\rme^{b_0}$. After $|x|\sqrt{n}$ generations, from the point of view of each particle, the test-set $\sqrt{n}A$ is shifted by 
$-x\sqrt{n}$ and therefore, in light of~\eqref{e:1.3}, the fraction of its descendants in $\sqrt{n}A$ at time $n$, will concentrate around $\nu(A - x)$.

Consequently, if it is possible to find $x \in \bbR$ such that $\nu(A - x) \geq p$, then employing the above shift strategy with $x$ for which $|x|$ is minimal, and then letting all particles evolve normally for the remaining $n-|x| \sqrt{n}$ generation, realizes the event $\{\ol{Z}_n(\sqrt{n}A) \geq p\}$. The ``probabilistic cost'' of such a shift-strategy is the cost of controlling all particles for the first $|x|\sqrt{n}$ generations, and since reproduction is kept at its minimum, the number of such particles is of order $\rme^{b_0 |x|\sqrt{n}}$. This gives a lower bound of order $\exp \big(\!-\rme^{b_0 |x| \sqrt{n}}\big)$, on a double exponential scale, and shows that such probabilities decay at most doubly exponentially with speed $\sqrt{n}$ and rate
\begin{equation}
\label{e:1.M17}
b_0 \min\{\wt{I}^+(A,p) \,,\, \wt{I}^+(A,p)\} \sqrt{n} \,,
\end{equation}
where $\wt{I}^\pm(A,p)$ are as in~\eqref{e:1.5}.

If there is no $x \in \bbR$ such that $\nu(A - x) \geq p$, then shifting all particles is not enough. In this case, one employs a {\em dilation strategy}, whereby all particles move together 
for $un$ generations, taking a $+1$ step half of the time and $-1$ otherwise, again reproducing at the minimal rate of $\rme^{b_0}$. After this time, there are $n(1-u)$ generations left, before the number of particles in $\sqrt{n}A$ is tested. Thanks to the CLT scaling in~\eqref{e:1.3}, from the point of view of each particle, the test-set is now dilated by $1/\sqrt{1-u}$. Employing the shift strategy from before on all particles for additional $x\sqrt{n}$ generations, shifts the set further by $x$, so that the overall fraction of descendants in $\sqrt{n}A$ at time $n$ now concentrates around $\nu(-x + A/\sqrt{1-u})$. 

As before, if $u$ and $x$ are such that $\nu(-x + A/\sqrt{1-u}) \geq p$, then the above will realize the event $\{\ol{Z}_n(\sqrt{n}A) \geq p\}$ at the cost of controlling an order of $\rme^{b_0 (un + |x|\sqrt{n})}$ particles. This gives a lower bound of $\exp \big(\!-\rme^{b_0 (u n + |x|\sqrt{n})}\big)$, on a double exponential scale. Taking the minimal possible $u > 0$, shows that such probabilities decay at most doubly exponentially with speed $n$ and rate $b_0 \wt{J}(A,p)$, where $\wt{J}(A,p)$ is as in~\eqref{e:1.6}.

Once the law of a general random point measure is allowed as the offspring distribution, it is no longer clear that the optimal way to realize the event $\{\ol{Z}_n(\sqrt{n}A) \geq p\}$ is via the shift and dilation strategies, as in the simple case above. Moreover, even if it is, it is not clear how to implement such shift and dilation strategies. For one, one cannot necessarily ``force'' all born particles to move to the same position, since such a configuration might not be in the support of $Z_1$. Even if there is such a configuration, it need not be one with a minimal number of particles, nor with the position taken by all particles maximal in magnitude.

Take Example~\ref{sss:2}, for instance. In order to implement a shift in the positive direction, one could naively use configuration $\rme^{b_1} \delta_{r^+}$ repeatedly. This, however, results in the population reproducing at the highest possible rate, and hence might (and will) not be the optimal way for a shift. Shifting in the case of Example~\ref{sss:3} seems even more difficult. Here, all particle configurations in the support of $Z_1$ are symmetric around $0$ and hence the average height of all born particles is the same as height of their parent.

Nevertheless, this work shows that the ``optimal'' way for realizing the event $\{\ol{Z}_n(\sqrt{n}A) \geq p\}$ in the general case, is still by shifting or dilating the test-set $A$, exactly as in the simple case. In particular, the probability of this event still decays doubly exponentially and with speed which is either $\sqrt{n}$ or $n$, depending on whether $\min\{\wt{I}^+(A,p), \wt{I}^-(A,p)\} < \infty$ or not. However, unlike in the simple case, the rate of decay is now also determined by the cost of implementing such shift and dilation strategies, which is not as straightforward as before.

The new ingredient which captures this ``implementation cost'' comes in the form of $\alpha^\pm$, as defined in~\eqref{e:7}. These quantities represent the {\em maximal} or, in light of Proposition~\ref{p:1.2b}, the {\em asymptotic rate of increase/decrease in the average particle height}. We show that such rate of height increase/decrease can be achieved asymptotically by an arbitrarily high fraction of the particles and at an arbitrarily small mean reproduction rate. Explicitly, we show that for any $\epsilon > 0$ and $n \geq 1$, there are configurations with $n$ generations, in which $1-\epsilon$ of the particles increase/decrease their heights at rate $\alpha^\pm(1-\epsilon)$, while reproducing on average at rate $\rme^{b_0 + \epsilon}$.

For the shift case, we can then use such configurations to shift $1-\epsilon$ fraction of the particles to $x \sqrt{n}$, in $|x|\sqrt{n}/(\alpha^{\text{\sgn}(x)}(1-\epsilon))$ many generations, while keeping the reproduction rate at $\rme^{b_0 + \epsilon}$, with $\epsilon > 0$ arbitrary small. Proceeding as in the simple case, we then get a lower bound of a double exponential order with speed $\sqrt{n}$ and rate $I(A,p)$ as in~\eqref{e:1.11a}. Similarly, for the dilation case, we can use such configurations to shift $1-\epsilon/2$ fraction of the particles up at rate $\alpha^+(1-\epsilon)$ for $un \times \alpha^-/(\alpha^++\alpha^-)$ 
generations and then $1-\epsilon/2$ fraction of the particles down at rate $\alpha^-(1-\epsilon)$ for 
$un \times \alpha^+/(\alpha^++\alpha^-)$ generations. This brings $1-\epsilon$ fraction of the particles to about $0$ in $un$ generations while keeping reproduction rate at $\rme^{b_0 + \epsilon}$. Proceeding essentially as in the simple case when $\min\{\wt{I}^+(A,p) \,,\, \wt{I}^-(A,p)\} = \infty$, we get a lower bound of a double exponential order, with speed $n$ and  rate $J(A,p)$ as in~\eqref{e:1.11a}. 

To obtain matching upper bounds, namely to show that these shift and dilation strategies are ``optimal'', the idea is to stop the process just before the needed strategy can be completed. That is after $(1-\delta) \wt{J}(A,p)n$ steps, if $I(A,p) = \infty$ and after $(1-\delta) \min\{\wt{I}^+(A,p)/\alpha^+\,,\,\wt{I}^-(A,p)/\alpha^-\} \sqrt{n}$ steps, if $I(A,p) < \infty$. In the first case, from the point of view of any particle, when the process is stopped the test-set is dilated by $1/\sqrt{1-u}$ for $u < \wt{J}(A,p)$.
This implies by the definition of $\wt{J}(A,p)$ in~\eqref{e:1.6}, that regardless of the particle's height, the mean fraction of its children which will lie in the test-set at time $n$ will be smaller than $p$. Since all particles evolve independently from this time on, standard large deviation inequalities (uniform exponential Chebyshev) can be used to show that the probability that the overall fraction of children is at least $p$, decays exponentially in the number of particles at the stopping time. As there are at least $\rme^{b_0 \wt{J}(A,p) (1-\delta)n}$ such particles, we obtain an upper bound on a double exponential scale of $\exp \big(\!-\!\rme^{J(A,p) (1-\delta)n}\big)$ for any $\delta >0$, as desired.

If $I(A,p) < \infty$, we wish to argue in the same way and claim that for almost all particles the mean fraction of children in the test-set is smaller than $p$, once the process is stopped as above. This will follow from the definition of $\wt{I}^\pm(A,p)$ in~\eqref{e:1.5}, if we can show that an asymptotic rate of height increase/decrease larger than $\alpha^\pm$ cannot be achieved, even if only $\epsilon$ fraction of the particles are considered, with $\epsilon>0$ arbitrarily small. We show that this statement indeed holds and, proceeding as in the case $I(A,p) = \infty$, this shows that $\exp \big(\!-\!\rme^{I(A,p) (1-\delta) \sqrt{n}}\big)$ is an upper bound on a double exponential scale, for any $\delta >0$.

\subsubsection{Maximal Rate of Increase/Decrease in Average Particle Height}
In light of the discussion in the previous subsection, it is not surprising that a substantial part of this work, and its main novelty, is the study of the quantities $\alpha^\pm$. The fact that $n \mapsto \alpha^\pm_n/n$ is increasing, is a standard consequence of the super-additivity property of the sequence $\alpha^\pm_n$. This, in turn, follows rather obviously from definition~\eqref{e:7} and the fact that the average particle height adds up under convolution: $\la \overline{\zeta * \zeta'}\,,\, x \ra = \la \ol{\zeta} \,,\, x \ra + \la \ol{\zeta}' \,,\, x \ra$. 

Thanks to the central limit theorem, not only does this show that the mean particle height increases linearly under repeated convolution of a particle configuration with itself, but also that the resulting empirical distribution is concentrated, in a Gaussian manner, around this mean. This readily implies that the maximal rate of increase/decrease in the average particle height is also (asymptotically) a lower bound on the maximal rate for increase/decrease in the height of $1-\epsilon$ fraction of the population for any $\epsilon > 0$. As mentioned, we prove that also the converse is true, namely that one cannot achieve a higher (asymptotic) rate of increase/decrease in height by considering only the top $\epsilon$ fractile of the population for any $\epsilon > 0$.

What is somehow less apparent at first sight, is that regardless of the reproduction law, one can achieve the optimal rate of height increase/decrease, using asymptotically the minimal possible average reproduction rate, $\rme^{b_0}$. Consider for instance example~\ref{sss:2}. It is clear that the highest rate of increase in height is $+r$ and that, one can easily have all particles increase their height at this rate, by repeatedly reproducing according to $\delta_{r+}$. It turns out, however, that one can still obtain the same asymptotic increase rate of $+r$ with an asymptotic mean reproduction rate of $\rme^{b_0}$.

As Example~\ref{sss:3} and the statement and proof of Proposition~\ref{p:1.2} demonstrate, it may be non-trivial to find the value of $\alpha^\pm$. There is no reason to believe that there should be a more explicit definition of $\alpha^\pm$ aside from via the limiting procedure in Proposition~\ref{p:1.2b} or in terms of the variational problem in~\eqref{e:7}. For example, one can devise reproduction laws in which obtaining a maximal average height requires both a composition of any number of configurations in $\cZ_1$ and a limiting procedure such as that in the proof of Proposition~\ref{p:1.2}. 

\subsubsection{Minimality of Assumptions}
\label{sss:Assumptions}
Let us now discuss the assumptions imposed on the process $Z$ and in particular argue that they are ``minimal'', namely that they are either necessary for Theorem~\ref{thm:1} (and Theorem~\ref{thm:1'}) to hold, or that if they are not satisfied, one can instead consider a modified process $Z'$, for which they are satisfied, and then easily translate the result back to $Z$. 

Assumption (A1), the so-called Kesten-Stigum condition, is necessary for both convergences in~\eqref{e:1} to hold (see~\cite{biggins1990central}). Without this condition~\eqref{e:1.3} need not hold in the general case and hence also the result of Theorem~\ref{thm:1}. Assumption (A2) is just centering and normalization of the average height after one generation, and by considering $Z_1 (\sigma \cdot + \mu)$, where $\mu = \bbE \la Z_1 \,,\, x \ra / \bbE |Z_1|$ and $\sigma^2 = \bbE \la Z_1 \,,\, (x - \mu)^2 \ra / \bbE|Z_1|$, any process $Z$ for which $\bbE \la  Z_1\,,\, x^2 + |x| \ra < \infty$ can be made to satisfy this assumption, with the result of the theorem translated accordingly. The latter moment condition for the intensity measure of $Z_1$ is indeed necessary for the second convergence in~\eqref{e:1} to hold (see~\cite{biggins1990central}), but in our case it is automatically satisfied because of Assumptions (A1) and (A4).

Assumption (A3) requires that the process $Z$ has at least binary branchings. Without this condition, it might be possible to implement a shift or dilation strategy, while having the number of particles grow only linearly. This will result in an exponential (and not double exponential) regime of decay for the probabilities in the theorem. Similarly, if the support of $Z_1$ is not bounded, as required by Assumption (A4), then one might be able to increase $\ol{Z}_n(\sqrt{n}A)$ by having particles reach $\sqrt{n} A$ in $o(\sqrt{n})$ generations (even one generation). This again, will result in a completely different regime of decay. The fact that the support of $Z_1$ lies in some deterministic discrete set almost surely, is made just for convenience and the proofs carry through without this assumption, albeit with some extra epsilons and deltas.

Lastly, we turn to Assumption (A5) which effectively requires that either the size of the population in the first generation, or its average height is random. Theorem~\ref{thm:1} is already non-trivial even if one of these is a constant. This is evident from Example~\ref{sss:1} with, e.g., $N \equiv 2$ (constant size) and Example~\ref{sss:3} (constant average height). To see the necessity of this requirement, observe that if both the size and average height are constant in the first generation, then in light of (A2), the average height will be $0$ in all generations and consequently $\alpha^\pm = 0$. Moreover, it takes a standard application of the Azuma-Hoeffding inequality (or martingale central limit theorem) to show that in this case:
\begin{prop}
\label{p:1.4}
Let $Z = (Z_n)_{n \geq 0}$ be the branching random walk process defined as in Subsection~\ref{ss:1} and satisfying Assumptions (A1)-(A4) but not (A5). Then for all $x > 0$ and $n \geq 1$,
\begin{equation}
\bbP \Big( \ol{Z}_n([-\sqrt{n}x \,,\, \sqrt{n}x]^\rmc\big) > 2\rme^{-x^2/(2 r_0)} \Big) = 0 \,.
\end{equation}
\end{prop}
In particular, for $A \in \cA$ sufficiently far from $0$ and $p$ sufficiently large, the probability of $\{\ol{Z}_n(\sqrt{n}A) \geq p\}$ will be $0$ for all $n \geq 1$.

\subsubsection{Organization of paper}
The remainder of the paper is organized as follows. Section~\ref{s:2}, is devoted to studying $\alpha^\pm_n$ and $\alpha^\pm$. Results in this sections are used crucially in the proof of the main theorem. This section also includes the proof of Proposition~\ref{p:1.2b} and Proposition~\ref{p:1.4}. Section~\ref{s:3} contains some technical statements, most-of-which were already stated and proved in~\cite{louidor2015large}. These include properties of the functions $\wt{I}(A,p)$ and $\wt{J}(A,p)$, results concerning the rate of convergence in the central limit theorem and uniform large deviation estimates. For the sake of self-containment and since they are rather short, proofs are provided. In Section~\ref{s:4}, the results of the previous two sections are used to prove the main theorem. Section~\ref{s:5} includes the proof of Proposition~\ref{p:1.2}. As usual, the letter $C$ and its decorated versions (e.g. $C'$) denote positive constants whose value may change from one line to another.

\section{Maximal Rate of Increase/Decrease in the Average Height}
\label{s:2}
In this section we study the maximal rate of increase or decrease in the average particle height, namely $\alpha^\pm$. Results from this section are then used in the proof of Theorem~\ref{thm:1}.

\subsection{Equivalence of Maximum and Limit}
We first show that the sequence $(\alpha^\pm_n/n :\: n \geq 1)$ is monotone increasing and hence that the maximization~\eqref{e:7} can be replaced by taking a limit. This is a straight-forward consequence of super-additivity of the sequence $\alpha^\pm_n$.

\begin{proof}[Proof of Proposition~\ref{p:1.2b}]
We shall prove the statement for $\alpha^+_n$ as the argument for $\alpha^-_n$ is similar, and begin by showing monotonicity. This is a straightforward consequence of super-additivity of the sequence $\alpha^+_n$. Indeed let $n_1, n_2 \geq 1$ and $\epsilon > 0$. By definition for $i=1,2$, we may find $\zeta^i \in \cZ_{n_i}$ such that $\la \ol{\zeta}^i \,,\, x \ra \geq \alpha^+_{n_i} - \epsilon$. Setting $\zeta := \zeta^1 * \zeta^2$, we clearly have $\zeta \in \cZ_{n_1+n_2}$ and $\la \ol{\zeta} \,,\, x \ra =  \la \ol{\zeta}^1 \,,\, x \ra + \la \ol{\zeta}^2 \,,\, x \ra \geq \alpha^+_{n_1} + \alpha^+_{n_2} - 2 \epsilon$. Since $\epsilon > 0$ is arbitrary, this shows that $\alpha^+_{n_1+ n_2} \geq \alpha^+_{n_1} + \alpha^+_{n_2}$. It is then a consequence of Fekete's lemma~\cite{fekete1923verteilung}, that $\alpha^+_n/n$ is monotone non-decreasing and hence converging to $\alpha^+$.

Thanks to Assumption (A4) we clearly must have $\alpha^+ \in [-r_0, r_0]$. To show that $\alpha^+$ is positive, suppose first that $\bbP(\la Z_1, x \ra = 0) < 1$, then by Assumption (A2), there must exist $\zeta_1 \in \cZ_1$ such that $\la \ol{\zeta}_1, x \ra > 0$, which shows that $\alpha^+ > 0$. Otherwise, by Assumption (A5), we must have $b_0 < b_1$. Moreover by (A2) again, there must exist $\zeta_1 \in \cZ_1$, such that $\la \zeta_1, x^2 \ra \neq 0$. It follows that there must exist $z > 0$ such that $z \in \supp(\zeta_1)$. Now choose also $\zeta \in \cZ_1(b_0)$ and $\zeta' \in \cZ_1(b')$ for some $b' \in (b_0, b_1] \cap \bbR$ and construct $\zeta_2$ from $\zeta_1$, by having a particle whose height in $\zeta_1$ is $z$ evolve one generation according to $\zeta'$, while all other particles in $\zeta_1$ evolve one generation according to $\zeta$. Then, $\la \zeta_2 \,,\, x \ra = \rme^{b_0} \la \zeta_1\,,\, x \ra + \big(\rme^{b'} - \rme^{b_0}\big) z = \big(\rme^{b'} - \rme^{b_0}\big) z > 0$, which shows that $\alpha^+ > 0$ since $\zeta_2 \in \cZ_2$.
\end{proof}

\subsection{Equivalence of Average and Fractiles}
In this subsection we show that the optimal asymptotic rate of increase/decrease in the average particle height is equal to the optimal asymptotic rate of increase/decrease in the height of a $t$ fraction of the particles for any $t \in (0,1)$. Moreover, there are particle configurations, which achieve such asymptotic rate of increase/decrease with $\rme^{b_0+\epsilon}$ average reproduction rate, for arbitrarily small $\epsilon >0$.  

We begin by a lemma which shows that the optimal asymptotic rate of increase/decrease in the average particle height cannot be beaten by considering only the height of an $\epsilon$ fraction of the heights, for any $\epsilon > 0$
\begin{lem}
\label{l:8}
For all $\epsilon > 0$, there exists $n_0$ large enough, such that for all $n > n_0$ and $\zeta \in \cZ_n$
\begin{equation}
\label{e:58}
\ol{\zeta} \Big(n\big[-\alpha^-\! - \epsilon \,,\, \alpha^++\epsilon \big]^\rmc\Big) \leq \epsilon
\quad. 
\end{equation}
\end{lem}

\begin{proof}
By the union bound, it is enough to show that for all $\epsilon > 0$, if $n$ is large enough then
\begin{equation}
\label{e:58.1}
\ol{\zeta} \big((-\infty \,,\, -n(\alpha^-+\epsilon)]\big) \leq \epsilon
\quad, \qquad
\ol{\zeta} \big([n(\alpha^++\epsilon) \,,\, \infty) \big) \leq \epsilon \,.
\end{equation}
Moreover, as the arguments are symmetric, we shall only treat the case of $\alpha^+$. The proof is different depending on whether $b_1 > b_0$ or $b_1 = b_0$. Suppose first that the former holds and choose $b_1' \in (b_0, \infty)$ (possibly $b_1$ itself) such that $\cZ_1(b'_1) \neq \emptyset$ and for any $m \geq 1$ also $\zeta_m \in \cZ_m(b_0)$ and $\zeta'_m \in \cZ_m(b_1')$.

Now let $n \geq 1$, $\epsilon > 0$ and $\zeta \in \cZ_n$ be such that the second inequality in~\eqref{e:58.1} does not hold. Then define $\zeta'$ from $\zeta$ by having all particles at or above $n(\alpha^+ + \epsilon)$ reproduce according to $\zeta'_m$, while all other particles reproduce according to $\zeta_m$. Since $\supp(\zeta'_m) \subseteq [-r_0 m, r_0 m]$ by Assumption (A4), we have
\begin{equation}
\zeta' \big( \big(\!-\!\infty ,\, n(\alpha^+ + \epsilon) - r_0 m \big) \big) 
\leq (1-\epsilon) |\zeta| \rme^{b_0 m} 
\quad, \qquad
|\zeta'| \geq \epsilon |\zeta| \rme^{b_1' m}  \,.
\end{equation}
This implies that 
\begin{equation}
\ol{\zeta}' \big( \big(\!-\!\infty ,\, n(\alpha^+ + \epsilon) - r_0 m \big) \big) 
\leq \frac{1-\epsilon}{\epsilon} \rme^{(b_0 - b'_1)m} =: \delta_\epsilon(m)\,.
\end{equation}
Observe that for any $\epsilon > 0$ we can make $\delta_\epsilon(m)$ arbitrarily small by choosing $m$ large enough. 

Using Assumption (A4) again we have $\supp(\zeta') \subseteq [-(n+m) r_0, (n+m) r_0]$ and consequently, 
\begin{equation}
\begin{split}
\la \ol{\zeta}' \,,\, x \ra & \geq \big(1-\delta_\epsilon(m) \big) \big(n(\alpha^+ + \epsilon) - r_0 m \big)
- \delta_\epsilon(m) (n+m) r_0\\
	& \geq n \big( (\alpha^+ + \epsilon) \big(1-\delta_\epsilon(m) \big) - \delta_\epsilon(m) r_0 \big) - r_0 m  \\
	& \geq n \big (\alpha^+ + \epsilon - 2 \delta_\epsilon(m) (r_0 + \epsilon) \big) - r_0 m \,,
\end{split}
\end{equation}
where we have used that $\alpha^+ \leq r_0$ to get to the last inequality. Therefore, for any $\epsilon > 0$, by choosing $m$ large enough and then $n$ large enough, we can ensure that
\begin{equation}
\la \ol{\zeta}', x \ra \geq (n+m) ( \alpha^+ + \epsilon /2) \,.
\end{equation}
Since $\zeta' \in \cZ_{n+m}$, this shows that $\alpha^+ \geq \alpha^+_{n+m}/(n+m) \geq \alpha^+ + \epsilon/2$, which is a contradiction. This shows~\eqref{e:58} in the case that $b_1 > b_0$.

Suppose now that $b_0 = b_1$ and let $\zeta_0, \dots, \zeta_n = \zeta$ for some $n \geq 1$ be a possible realization of the branching random walk in $n$ generations. By definition, 
\begin{equation}
\zeta_m = \sum_{x \in \zeta_{m-1}} \zeta_{m-1}^{(x)}(\cdot - x)
\end{equation}
where $\zeta_{m-1}^{(x)} \in \cZ_1(b_0)$ for all $x \in \zeta_{m-1}$. Observe that this already shows that $\alpha^+_n/n = \alpha^+_1 = \alpha^+$ for all $n \geq 1$. Then on a different probability space with measure $P$ and expectation $E$, define the random variables $X_0, X_1, \dots, X_n$ inductively such that $X_0 = 0$ and $X_k$ for $k = 1, \dots, n$ is the height of a particle in $\zeta_k$ uniformly chosen from all particles whose parent in $\zeta_{k-1}$ has height $X_{k-1}$.

Now, by construction $X_m$ has distribution $\ol{\zeta}_m$ under $P$, for $m=0, \dots, n$ and moreover clearly, $|X_m - X_{m-1}| \leq r_0$ and $E \big(X_m - X_{m-1}\,|\, X_{m-1}\big) \leq \alpha^+$ for $m=1, \dots, n$.
Then by the Azuma-Hoeffding inequality for the sub-martingale $\big(X_m - m \alpha^+ :\: m=0, \dots, n\big)$ (with respect to the natural filtration), for any $\epsilon > 0$, 
\begin{equation}
\nonumber
\ol{\zeta} \big(\big[n(\alpha^+ + \epsilon), \infty\big)\big) 
= P \big( X_n > n(\alpha^+ + \epsilon)  \big)
= P \big( X_n - n\alpha^+ > n\epsilon  \big) 
\leq \rme^{-(\epsilon n)^2/(8r_0^2 n)} \,,
\end{equation}
which will be come smaller than $\epsilon$ once $n$ is large enough. 
\end{proof}

The last part of the proof, can be reused to prove Proposition~\ref{p:1.4}.
\begin{proof}[Proof of Proposition~\ref{p:1.4}]
If Assumption (A5) does not hold, then almost surely, $|Z_1| = b_0 = b_1$ and $\la Z_1 \,,\, x \ra = 0$. Repeating the argument from the previous proof for the case $b_0 = b_1$, we see that $\alpha^\pm = 0$ and that in particular $(X_n)_{n \geq 0}$ as constructed in the proof, is a martingale with increments bounded by $r_0$. It follows via the Azuma-Hoeffding inequality, exactly as before, that
\begin{equation}
\ol{\zeta}_n \big(\big[-\sqrt{n}x \,,\, \sqrt{n}x\big]^\rmc\big) 
= P \big( |X_n| > \sqrt{n}x \big) 
\leq 2\rme^{-x^2/(2 r_0^2)} \,,
\end{equation}
for any $\zeta_n \in \cZ_n$ and $n \geq 1$. 
\end{proof}

Next we show that the optimal asymptotic rate of increase/decrease in the average particle height can be achieved by $1-\epsilon$ fraction of the particles and using configurations with an asymptotic reproduction rate of $\rme^{b_0 + \epsilon}$, for arbitrary small $\epsilon >0$. This gives the opposite direction to the statement of Lemma~\ref{l:8}.

\begin{lem}
\label{l:3.2}
For all $\epsilon > 0$, there exist $n_0 > 0$ such that for all $n \geq n_0$ we may find $\zeta^+, \zeta^- \in \cZ_n$ satisfying,
\begin{equation}
\label{e:M67}
\ol{\zeta}^+ \big(n[\alpha^+-\epsilon\,,\, \alpha^++\epsilon]\big) \geq 1 - \epsilon
\ , \quad
n^{-1} \log |\zeta^+| \leq b_0 + \epsilon
\ , \quad
\bbP(Z_n = \zeta^+) \geq \rme^{-\rme^{(b_0 + \epsilon)n}} 
\end{equation}
and
\begin{equation}
\label{e:M68}
\ol{\zeta}^- \big(-n[\alpha^--\epsilon\,,\, \alpha^-+\epsilon]\big) \geq 1 - \epsilon
\ , \quad
n^{-1} \log |\zeta^-| \leq b_0 + \epsilon
\ , \quad
\bbP(Z_n = \zeta^-) \geq \rme^{-\rme^{(b_0 + \epsilon)n}} \,.
\end{equation}
\end{lem}
\begin{proof}
We shall prove only the existence of $\zeta^+$ satisfying~\eqref{e:M67} as the proof for the second statement is essentially identical. Thanks to Lemma~\ref{l:8}, it is enough to show~\eqref{e:M67} with $[\alpha^+-\epsilon\,,\,\alpha^++\epsilon]$ replaced by $[\alpha^+-\epsilon\,,\,\infty)$ in the first inequality. Now, fix any $\epsilon > 0$ and use Proposition~\ref{p:1.2b} to find $m$ large enough and $\xi_m \in \cZ_m$ such that $\la \ol{\xi}_m \,,\, x \ra \in m [\alpha^+ - \epsilon \,,\,\, \alpha^+]$. Then, fix also $n \geq 1$ and find integers $k \geq 0$, $r = 1, \dots, m-1$ such that $n = mk + r$. Finally, let $\xi_1 \in \cZ_1(b_0)$ be an arbitrary but fixed particle configuration.

Suppose first that $|\xi_m| = \rme^{(b_0 + \delta) m}$ for some $\delta > 0$. In this case, we decompose $\xi_m$ as $\xi_m^g + \xi_m^b$, where $\xi_m^g$ contains the top $\lceil \rme^{(b_0 + \epsilon/2)m} \rceil$ particles from $\xi_m$ (possibly all of them, if $|\xi_m|$ is not large enough) and $\xi_m^b$ contains the rest. Next we construct $\zeta_{mj} = \zeta_{mj}^g + \zeta_{mj}^b$ for $j \geq 0$  inductively as follows. For generation $0$ we set $\zeta_0^g = \delta_0$ and $\zeta_0^b = 0$. Then to get to generation $m(j+1)$ from generation $mj$ we set
\begin{equation}
\zeta_{m(j+1)}^g = \zeta_{mj}^g * \xi_m^g
\quad, \qquad
\zeta_{m(j+1)}^b = \zeta_{mj}^g * \xi_m^b + \zeta_{mj}^b * (\xi_1)^{*m}
\end{equation}
In words, the ``good'' particles from generation $mj$ are reproduced according to $\xi_m$, while the ``bad'' particles from generation $mj$ are reproduced according to $\xi_1$ for $m$ generations. Then the ``good'' particles in generation $m(j+1)$ are the top $\lceil \rme^{(b_0 + \epsilon/2)m} \rceil$ children of each of the good particles in generation $mj$, while the ``bad'' particles in generation $m(j+1)$ are all the remaining children.

From the construction, it follows with $\delta' := \min\{\epsilon/2,\delta\} > 0$ that,
\begin{equation}
\label{e:M3.36}
\big|\zeta_{mj}^g\big| = |\zeta_m^g|^j \in \big[\rme^{(b_0+\delta') mj} \,,\,\, 
\rme^{(b_0+\epsilon)mj} \big] \,.
\end{equation}
Moreover, the number of particles in $\zeta_{mk}^b$ which are descendants of particles in $\zeta_{mj}^g$ for $j \leq l < k$ is  
\begin{equation}
\label{e:M3.37}
\zeta_{mk}^{b,\leq l } := 
\sum_{j=0}^{l} |\zeta_{mj}^g| |\xi_m^b| |\xi_1|^{m(k-j-1)}
\leq |\zeta_{mk}^g||\xi_m^b|\rme^{-b_0 m} \sum_{j=0}^{l} \rme^{ -\delta' m(k-j)}
\leq C |\zeta_{mk}^g| \rme^{-\delta' m (k-l)} \,,
\end{equation}
with $C > 0$ depending on $m$ and $\zeta_m$. In particular, plugging in $l=k-1$, we have 
\begin{equation}
\label{e:M3.38}
|\zeta_{mk}| = |\zeta_{mk}^g| + |\zeta_{mk}^b| \leq (1+C) \rme^{(b_0+\epsilon) mk} \,.
\end{equation}.

In addition, by construction of $\xi^g_m$ we have  $\la \ol{\xi}_m^g \,,\, x \ra \geq \la \ol{\xi}_m \,,\, x \ra \geq m (\alpha^+ - \epsilon)$. It then follows from the central limit theorem applied to the distribution $\ol{\xi}_m^g$, that for all $l$ large enough,
\begin{equation}
\label{e:M3.39}
\ol{\zeta}_{ml}^g \big (\!-\!\infty, (\alpha^+ - 2\epsilon) ml \big) \leq \epsilon \,.
\end{equation}
Since within $(k-l)m$ generations, a particle's height can decrease by at most $r_0 m(k-l)$, the above shows, in particular, that for all $l \leq j < k$, the fraction of particles in $\zeta_{mk}^b$ which are descendants of $\zeta_{mj}^g$ and reach $\big(\!-\!\infty \,,\, (\alpha^+ - 2\epsilon) ml - r_0(k-l)m\big)$ at generation $mk$ is at most $\epsilon$. Using also~\eqref{e:M3.39} with $l=k$ we arrive to
\begin{equation}
\label{e:M3.40}
\zeta_{mk} \Big(\big(\!-\!\infty \,,\, ml (\alpha^+ - 2\epsilon) - r_0 m(k-l) \big)\Big)
\leq \big|\zeta_{mk}^{b, \leq l}\big| + \epsilon \big(\big|\zeta_{mk}\big| - \big|\zeta_{mk}^{b, \leq l}\big|\big) \,.
\end{equation}
Dividing by $|\zeta_{mk}|$ and using~\eqref{e:M3.37} yields
\begin{equation}
\ol{\zeta}_{mk} \Big(\big(\!-\!\infty \,,\, ml (\alpha^+ - 2\epsilon) - r_0 m(k-l) \big)\Big)
\leq C \rme^{-\delta' m (k-l)} + \epsilon \,.
\end{equation}
Then there exists $s \geq 0$, such that with $l := k-s$ and all $k$ large enough, the right hand side above is at most $2 \epsilon$, while the left hand side at least $\ol{\zeta}_{mk} \big(\!-\!\infty \,,\, mk(\alpha^+-3\epsilon) \big)$.
 
We now recall that $n=mk + r$ and therefore we can construct $\zeta^+ \in \cZ_n$ from $\zeta_{mk}$ by setting $\zeta := \zeta_{mk} * (\xi_1)^{*r}$. In light of the above and the fact that $\supp (\xi_1^{*r}) \in [-r r_0,\, r r_0]$, we can write for all $n$ large enough,
\begin{equation}
\ol{\zeta}^+\big((- \infty \,,\, n(\alpha^+-4\epsilon)\big) \leq 
\ol{\zeta}_{mk} \big((- \infty \,,\,  mk(\alpha^+-3\epsilon) \big) \leq 2\epsilon \,.
\end{equation}
Also for large $n$ from~\eqref{e:M3.38} we get that $|\zeta^+| \leq C \rme^{(b_0 + \epsilon)mk + b_0 r} \leq \rme^{(b_0 + 2\epsilon) n}$.
Finally, the only configurations involved in the construction of $\zeta^+$ are $\xi_1$ and $\xi_m$ which are fixed for any given $\epsilon$ and are obtained with positive probability by $Z_1$ and $Z_m$ respectively. Since at most $n|\zeta^+|$ particles reproduce to obtain $\zeta^+$, it follows that 
\begin{equation}
\bbP(Z_n = \zeta^+) \geq C^{n |\zeta^+|} \geq \rme^{-\rme^{(b_0+4\epsilon)n}} \,,
\end{equation}
as long as $n$ is large enough. Choosing $\epsilon/4$ instead of $\epsilon$ in the first place yields~\eqref{e:M67}.

If $|\xi_m|=\rme^{b_0 m}$, namely when $\delta=0$, one can still use the above argument, albeit with $\xi_m^b = 0$ and consequently with $\zeta_{mk}^b = 0$, $\zeta_{mk} = \zeta_{mk}^g = \xi_m^{*k}$ and $\zeta^+ = \zeta_{mk} * (\zeta_1)^{*r}$. In particular~\eqref{e:M3.36} and~\eqref{e:M3.39} with $l=k$ can be directly applied to $\zeta_{mk}$.
\end{proof}

Next, we show that one can have an arbitrarily high fraction of the particles arbitrarily close to height $0$, while using configurations with asymptotically minimal mean reproduction rate. This will be used in the dilation strategy.
\begin{lem}
\label{l:11a}
For all $\epsilon > 0$ small enough, there exists $n_0 > 0$, such that for all $n > n_0$, we may find $\zeta \in \cZ_n$ satisfying 
\begin{equation}
\label{e:67.5}
\ol{\zeta} \big([-n \epsilon \,,\, n\epsilon]\big) \geq 1 - \epsilon
\quad, \qquad
n^{-1} \log |\zeta| \leq b_0 + \epsilon 
\quad, \qquad
\bbP (Z_n = \zeta) \geq \rme^{-\rme^{n(b_0+\epsilon)}} \,.
\end{equation}
\end{lem}

\begin{proof}
Let $\epsilon > 0$ and for any $n \geq 0$ set $n^+ := \lfloor \alpha^-/(\alpha^+ + \alpha^-) \rfloor n$ and $n^- := n - n^+$. Then for all $n$ large enough, we shall have
\begin{equation}
\bigg| \frac{n^+ \alpha^+ - n^- \alpha ^-}{n} \bigg| < \epsilon \,.
\end{equation}
Thanks to Lemma~\ref{l:3.2}, if $n$ is large enough we can find $\zeta^+ \in \cZ_{n^+}$ and $\zeta^- \in \cZ_{n^-}$ such that~\eqref{e:M67},~\eqref{e:M68} hold with $n$ replaced by $n^+$ and $n^-$ respectively. Now set $\zeta := \zeta^+ * \zeta^-$ and observe that $\zeta \in \cZ_n$ and that by the union bound, 
\begin{equation}
\ol{\zeta} \big([-2n \epsilon \,,\, 2n\epsilon]^\rmc \big)
\leq \ol{\zeta}^+ \big(n^+[\alpha^+-\epsilon \,,\, \alpha^++\epsilon]^\rmc \big)
+
 \ol{\zeta}^- \big(-n^-[\alpha^--\epsilon \,,\, \alpha^-+\epsilon]^\rmc \big)
\leq 2 \epsilon \,.
\end{equation}
Since $\log |\zeta| = \log |\zeta^+| + \log |\zeta^-| \leq (n^+ + n^-) (b_0 + \epsilon) \leq n (b_0 + \epsilon)$ and
\begin{equation}
\bbP (Z_n = \zeta) \geq \bbP \big(Z_{n^+} = \zeta^+) \big(\bbP \big(Z_{n^-} = \zeta^-)\big)^{|Z_{n^+}|} \geq \rme^{-\rme^{n (b_0 + 2\epsilon)}} \,,
\end{equation}
we can rename $\epsilon$ to $2\epsilon$ and claim the result.
\end{proof}

Finally, we need the following ``steering'' result, which shows that an arbitrarily high fraction of the descendants of a particles can be shifted together at a prescribed linear speed and stay arbitrarily close to each other.
\begin{lem}
\label{l:10a}
For all $\eta > 0$, we may find $r \in (0, r_0]$ and $b \in [b_0, b_1] \cap \bbR$ for which the following holds. For any $\epsilon > 0$, there exists $n_0 > 0$ such that for all $x \in (-r, r)$, $n > n_0$, there is $\zeta \in \cZ_n$ satisfying,
\begin{equation}
\label{e:71.5}
\ol{\zeta}\Big(\big[xn - n^{\eta} ,\, xn + n^{\eta} \big]\Big) \geq 1-\epsilon 
\quad , \qquad
n^{-1} \log |Z_n| \leq b
\quad , \qquad
\bbP \big(Z_n = \zeta) \geq \rme^{-\rme^{2bn}} 
\,.
\end{equation}
\end{lem}
\begin{proof}
We start by claiming the statement of the lemma with $\eta = 1/2 + \delta$ for any choice of $\delta>0$ and then iterate the argument to extend the result to arbitrarily small $\eta >0$. To this end, suppose first that $x=0$. In this case by~\eqref{e:1}, for any $\epsilon > 0$ as long as $n$ is large enough we have
\begin{equation}
\label{e:3.23}
\bbP \Big( \ol{Z}_n \big([-n^{1/2 + \delta} , n^{1/2+\delta}]\big) > 1-\epsilon 
\,,\,\, 
|Z_n| \leq \rme^{bn} \Big) 
\geq C > 0\,.
\end{equation}
with $b := 1 + \log \bbE |Z_1|$ and some $C > 0$, which depends only on the law of $Z_1$. This clearly implies~\eqref{e:71.5}

Suppose next that $x > 0$, let $\epsilon > 0$ be given and $n \geq 1$. By the uniformity in $x$ in the statement of the lemma, we may assume without loss of generality that $x > n^{-2/3}$. Since $\alpha^+ > 0$ as shown by Proposition~\ref{p:1.2b}, there exists $m \geq 1$, $0 < u \leq m r_0$, $b' \in [b_0, b_1] \cap \bbR$ and $\zeta'_m \in \cZ_m(b')$ such that $\la \ol{\zeta}'_m ,\, z \ra = u$. We then set
\begin{equation}
k := \lceil (xn/u)  \rceil 
\quad,\qquad
l := n - km \,,
\end{equation}
and observe that if $r = u/2m$ and
$n^{-2/3} < x \leq r$, then 
\begin{equation}
n^{1/3} \leq k \leq n 
\quad, \qquad
n/2 - m \leq l \leq n \,.
\end{equation}

Letting $\zeta'_{km} \in \cZ_{km}$ be the $k$-fold convolution of $\zeta'_m$ with itself, by the central limit theorem, for all $n$ large enough (depending only on $\epsilon$ and $\delta$),
\begin{equation}
\ol{\zeta}'_{km} \big( [xn-n^{1/2+\delta} \,,\, xn+n^{1/2+\delta}] \big)
\geq 
\ol{\zeta}'_{km} \big( [ku- k^{1/2+\delta/2} \,,\, ku+k^{1/2+\delta/2}] \big)
\geq 1-\epsilon \,,
\end{equation}
uniformly in all $x$ as above. At the same time, as in the case of $x=0$, if $n$ is large enough (depending only on $\epsilon$ and $\delta$) we can find $\zeta'_{l} \in \cZ_{l}$ such that 
\begin{equation}
\nonumber
\ol{\xi}'_{l} \big( [-n^{1/2+\delta}, n^{1/2+\delta}] \big)
\geq 
\ol{\xi}'_l \big( [-l^{1/2+\delta/2}, l^{1/2+\delta/2}] \big)
\geq 1-\epsilon 
\ , \quad
|\zeta'_l| \leq \rme^{bl}
\ , \quad
\bbP(Z_l = \zeta'_l) \geq C  > 0\,,
\end{equation}
with $b$ and $C$ as in~\eqref{e:3.23}.

We now set $\zeta := \zeta'_{km} * \zeta'_l$ and observe that $\zeta \in \cZ_n$ and  by the union bound we must have
\begin{equation}
\begin{split}
\ol{\zeta} \Big(\big[xn & - n^{1/2+2\delta} ,\, xn + n^{1/2+2\delta} \big]^\rmc \Big)  \\
& \leq \ol{\zeta'} \Big(\big[xn - n^{1/2+\delta},\, xn + n^{1/2+\delta} \big]^\rmc\Big) 
+\ol{\zeta'} \Big(\big[-n^{1/2+\delta},\, n^{1/2+\delta} \big]^\rmc\Big) 
\leq 2 \epsilon \,.
\end{split}
\end{equation}
for all $n$ large enough, uniformly in $x$ as above. Moreover, clearly $|\zeta| \leq \rme^{b'km + bl} \leq \rme^{b''n}$ with $b'' = \max\{b,b'\}$ and with such $b''$ we also have,
\begin{equation}
\bbP(Z_n = \zeta) \geq \prod_{i=1}^k \big(\bbP (Z_m = \zeta'_m)\big)^{\rme^{(i-1)mb'}}
	\big(\bbP(Z_l = \zeta'_l)\big)^{\rme^{km b'}}
\geq C^{(k+1) \rme^{kmb'}} \geq \rme^{-\rme^{2nb''}} \,,
\end{equation}
again, as along as $n$ is large. Using $2\epsilon$ and $2\delta'$ in place of $\epsilon$ and $\delta$,  this shows the statement of the lemma with $\eta = 1/2 + \delta$ and all $x \in [0, r)$ where $r=u/2m$. A symmetric argument extends this to all $x \in (-r, r)$ for a properly chosen $r > 0$.

We shall now show that, given some $l \geq 1$, if the statement of the lemma holds for all $\eta > 2^{-l}$, then it will also hold for all $\eta > 2^{-(l+1)}$. A simple induction argument will then finish the proof. Accordingly, fix any $l \geq 1$, $\delta > 0$ and suppose that the lemma indeed holds with $\eta = 2^{-l} + \delta$ and some $r, b \in (0,\infty)$. Let also $\epsilon > 0$, $x \in (-r/2, r/2)$ be given and $n \geq 1$. Setting $m := \lfloor n^{2^{-l} + 2 \delta} \rfloor$ and $n' = n-m$, by our assumption, if $n$ is large enough, then there exists $\zeta_{n'} \in \cZ_{n'}$ such that~\eqref{e:71.5} holds with $\eta=2^{-l} + \delta$ and $n'$, $\zeta_{n'}$ in place of $n$, $\zeta$ resp.

We now construct $\zeta$ from $\zeta_{n'}$ by letting any particle $y \in \zeta_{n'}$ satisfying $y \in \big[xn' - (n')^{2^{-l}+\delta} ,\, xn' + (n')^{2^{-l}+\delta}\big]$ reproduce in $m$ generations according to $\zeta^{(y)}_m \in \cZ_m$, where
\begin{equation}
\label{e:80}
\ol{\zeta}^{(y)}_m \Big( \big[xn - y - m^{1/2 + \delta} \,,\, xn - y + m^{1/2 + \delta} \big] \Big)
\geq 1-\epsilon 
\ \ , \ \ \ 
|\zeta^{(y)}_m| \leq \rme^{bm}
\ \ , \ \ \ 
\bbP \big(Z_m = \zeta^{(y)}_m\big) \geq \rme^{-\rme^{2bm}} \,,
\end{equation}
while all other particles reproduce according to $\zeta_m = (\zeta_1)^{*m}$ where $\zeta_1 \in \cZ_1(b_0)$ is an arbitrary but fixed particle configuration.
This is possible, once $n$ is large enough by our assumption and since
\begin{equation}
\big|xn-y \big| \leq \big| xn - xn' \big| + (n')^{2^{-l} + \delta}
< rm  \,,
\end{equation}
uniformly in $y$ and $x$.

But then $\zeta \in \cZ_n$ and by~\eqref{e:71.5} (with $n', \zeta'$, $\eta = 2^{-l} + \delta$) and~\eqref{e:80}, we have
\begin{equation}
\begin{split}
\ol{\zeta} \big( \big[xn - m^{1/2+\delta} \,,\, xn + m^{1/2+\delta} \big]^\rmc \big) 
& \leq \frac{ \epsilon |\zeta_{n'}| \rme^{b_0 m} + \big \la \zeta_{n'} \,,\, \epsilon |\zeta^{(\cdot)}_m| 
	1_{\{|\cdot-xn| \leq (n')^{2^{-l} + \delta}\}} \big \ra}{|\zeta|} \\
& \leq \frac{\epsilon |\zeta_{n'}| \rme^{b_0 m}}{|\zeta_{n'}| \rme^{b_0 m}}
	+ \frac{\epsilon |\zeta|}{|\zeta|} \leq 2\epsilon \,. 
\end{split}
\end{equation}
Since $m^{1/2 + \delta} \leq n^{2^{-(l+1)} + 2 \delta}$ if $\delta$ is small, this shows the first inequality in~\eqref{e:71.5} with $2 \epsilon$ in place of $\epsilon$ and $\eta = 2^{-(l+1)} + 2 \delta$. At the same time, from the construction it follows that for all $n$ large enough,
\begin{equation} 
|\zeta| \leq |\zeta_{n'}| \rme^{bm} \leq \rme^{bn}
\quad, \qquad
\bbP (Z_n = \zeta) \geq \bbP (Z_{n'} = \zeta_{n'}) \Big(\min \big\{ \rme^{-\rme^{2bm}},
	\rme^{-C m \rme^{b_0 m}} \big\}\Big)^{|\zeta_{n'}|}
\geq \rme^{-\rme^{2bn}} \,.
\end{equation}
This shows that the remaining inequalities in~\eqref{e:71.5} hold altogether yields the statement of the lemma for any $\eta > 2^{-(l+1)}$, as desired.
\end{proof}

\section{Technical Lemmas}
\label{s:3}
In this subsection we state some auxiliary lemmas, which will be used in the proof of the main theorem. Similar versions of these lemmas were stated and proved in~\cite{louidor2015large}. 
The first lemma deals with continuity of $\nu$, $\wt{I}^\pm$ and $\wt{J}$.
\begin{lem}
\label{l:10}
Let $A \in \cA$ be non-empty and $p \in (0,1)$.
\begin{enumerate}
  \item
    $(\rho, \zeta) \mapsto \nu(\rho A + \zeta) \in \rmC^{\infty}(\bbR^2)$.
  \item
    If $\wt{I}^\pm(A,p) \in [0, \infty)$ then with $x = \wt{I}^\pm(A,p)$ we have,
    \begin{equation}
      \nu(A \mp x) \geq p \, .
    \end{equation}
  In particular, if $p > \nu(A)$, then $\wt{I}^\pm(A,p) > 0$. 
  \item
    $\wt{J}(A,p) \in [0,1)$ and there exists $x \in \bbR$ such that with $u=\tilde{J}(A,p)$ we have,
    \begin{equation}
      \label{e:3}
      \nu \big( \big(A/\sqrt{1-u} \big) - x\big) \geq p \, .
    \end{equation}
  \item $\wt{J}(A,p) > 0$ if and only if $\min \{\wt{I}^+(A,p),\, \wt{I}^-(A,p)\} = \infty$. 
\end{enumerate}
\end{lem}

\begin{proof}
Part~1 follows from the dominated convergence theorem and standard arguments once we write 
\begin{equation}
\nu(\rho A + \zeta) = \int_A \frac{1}{\sqrt{2\pi}} e^{-\frac{(\rho t+\zeta)^2}{2}} \rho \rmd t
\end{equation}
since the integrand is in ${\rm C}^{\infty}(\bbR^2)$.

For part~2 and~3, if $A=\bbR$, then $\wt{I}^\pm(A,p) = \wt{J}(A,p) = 0$, and there is nothing to prove. Otherwise, define
\begin{equation}
  \varphi_A(u,x) = \nu\big( \big(A/\sqrt{1-u}\big) - x \big)
\end{equation}
which is in $\rmC^\infty([0,1) \times \bbR)$ by part~1.
Therefore $\{x \in \bbR_+ :\: \varphi_A(0, \pm x) \geq p\}$ is closed and bounded from below. Such set must contain a minimal value. This shows part~2.

For part~3, if $A$ contains a half-infinite interval, then since $\varphi_A(0, x) \to 1 > p$ if $x \to +\infty$ or $x \to -\infty$, we must have $\min \{\wt{I}^+(A,p) ,\, \wt{I}^-(A,p)\} < \infty$. Therefore $\wt{J}(A,p) = 0$ and \eqref{e:3} is satisfied 
with $u=0$ and $x = \min \{\wt{I}^+(A,p), \wt{I}^-(A,p)\}$, in light of part~2. Otherwise, $A$ is a finite union of finite intervals, and so there must exist $U < 1$, $R < \infty$ such that
$\varphi_A(u,x) \geq p$ for some $0 \leq u \leq U$ and $x$ with $|x| \leq R$ and $\varphi_A(u,x) < p/2$ for all $0 \leq u \leq U$ and $x$ with $|x| > R$. Thus, $\wt{J}(A,p)$ is the infimum of the continuous function $u$ over the 
non-empty compact set 
\begin{equation}
  \{(u,x) : \varphi_A(u,x) \geq p, \, 0 \leq u \leq U,\, |x| \leq R \} \,,
\end{equation}
which gives part~3.

As for part~4, it follows immediately from the definition that $\min\{\wt{I}^+(A,p),\, \wt{I}^-(A,p)\} < \infty$, implies $\wt{J}(A,p) = 0$. The other direction is a direct consequence of part~3 of the lemma. 
\end{proof}

For what follows, let us define $\nu_n$ for $n \geq 1$ as the intensity measure corresponding to $Z_n / \bbE Z_n$. That is, for any Borel set $A \subseteq \bbR$, we have
\begin{equation}
\label{e:3.6}
\nu_n(A) := \frac{\bbE Z_n(A)}{\bbE |Z_n|} \,.
\end{equation}
Observe that $\nu_n$ is the $n$-fold convolution of $\nu_1$ with itself and that by Assumption (A2) the latter is a probability distribution with mean zero and variance one. It therefore follows from the central limit theorem that $\nu_n(\sqrt{n} \cdot)$ converges weakly to the standard Gaussian measure $\nu$. Uniformity of this convergence is then given by,
\begin{lem}
\label{l:11}
Let $A \subseteq \bbR$ be a continuity set of $\nu(A)$, i.e. $\nu(\partial A) = 0$ and 
$R > 0$. Then,
\begin{equation}
  \lim_{n \to \infty} \sup_{\rho \in [R^{-1}, R]} \sup_{\zeta \in \bbR}
    \big|\nu_n(\sqrt{n}(\rho A + \zeta)) - \nu(\rho A + \zeta)\big| = 0 \,.
\end{equation}
\end{lem}
\begin{proof}
By Theorem 2 in~\cite{BillFlem}, it is enough to check that 
\begin{equation}
\label{eqn:20.2}
  \lim_{\delta \to 0} \sup_{\zeta, \rho} \nu\big( (\partial(\rho A + \zeta))^\delta \big) = 0,
\end{equation}
where for a set $D \subset \bbR$, we set $D^{\delta} = \{x \in \bbR :\: \inf_{y \in D} |x-y| < \delta\}$ and the supremum is over $\rho$ and $\zeta$ as in the statement in the lemma. Since $\nu$ is equivalent to $\lambda$, Lebesgue measure on $\bbR$, 
we may show~\eqref{eqn:20.2} with $\lambda$ in place of $\nu$. But,
\begin{equation}
  \lambda \big((\partial(\rho A + \zeta))^\delta \big) = 
  \lambda \big(\rho(\partial A)^{\delta/\rho} + \zeta \big) 
  \leq R \lambda \big( (\partial A)^{R \delta} \big) \,,
\end{equation}
The last term goes to $0$ as $\delta \to 0$, since $\lambda(\partial A) = 0$.
\end{proof}

We shall need the following uniform Chernoff-Cram\'{e}r-type upper bound. 
\begin{lem}
\label{l:12}
Let $\bbX$ be a family of random variables on $\bbR$ with zero mean such that for some
$\theta_0 > 0$
\begin{equation}
\label{eqn:1a}
  \sup_{X \in \bbX} \bbE e^{\theta_0 X} < \infty \quad \text{ and } \quad
  \sup_{X \in \bbX} \bbE (X^-)^2 < \infty \, .
\end{equation}
Then there exists $C > 0$ such that for any $x > 0$ small enough,
any $m\geq 1$ and $X_1, \dots, X_m$ independent copies of random variables in $\bbX$
\begin{equation}
\label{eqn:2}
  \bbP \big(\tfrac{1}{m} \sum_{i=1}^m X_i \, > \, x \big) \leq e^{-C x^2 m}
\end{equation}
\end{lem}

\begin{proof}
Using the exponential Chebyshev's inequality we may bound the l.h.s. in \eqref{eqn:2}
for any $0 < \theta \leq \theta_1 < \theta_0$ by
\begin{equation}
\label{eqn:3}
  \exp \big\{-m \big(u \theta - m^{-1} \sum_{i=1}^m L_{X_i}(\theta) \big) \big\} \, ,
\end{equation}
where we use $L_X(\theta) = \log \bbE e^{\theta X}$ for the log moment generating 
function of $X$. Since $L_X(\theta)$ is in $\rmC^{\infty}([0,\theta_0))$ due to \eqref{eqn:1a}, we may use Taylor expansion to write (note that the first two terms are $0$)
\begin{equation}
  L_{X}(\theta) = \tfrac12 L^{\prime\prime}_{X}(\wh{\theta}) \theta^2 \, ,
\end{equation}
for some $\wh{\theta} \in (0, \theta)$. Now if we denote by $M_X(\wh{\theta}) = \bbE e^{\wh{\theta} X}$ the moment generating function of $X$ then 
\begin{equation}
  L^{\prime\prime}_{X}(\wh{\theta}) = 
    \frac{M^{\prime\prime}_{X}(\wh{\theta})M_{X}(\wh{\theta}) - (M^\prime_{X}(\wh{\theta}))^2}{M^2_{X}(\wh{\theta})}
    < C M_{X}(\theta_0) + \bbE (X^-)^2\, .
\end{equation}
This follows since $M_{X}(\wh{\theta}) \geq 1$ via Jensen's inequality and since
\begin{equation}
  M^{\prime\prime}_{X}(\wh{\theta}) = \bbE X^2 e^{\wh{\theta} X} 
    \leq \bbE X^2 1_{X < 0} + C \bbE e^{\theta_0 X} 1_{X \geq 0} \,,
\end{equation}
for some $C>0$ independent of $X \in \bbX$. Therefore \eqref{eqn:1a} implies that there exists $K > 0$ for which
\begin{equation}
  \sup_{X \in \bbX} \sup_{\wh{\theta} \in (0, \theta_1)} L^{\prime\prime}_{X}(\wh{\theta}) < K \,
\end{equation}
and thus
\begin{equation}
  u \theta - m^{-1} \sum_{i=1}^m L_{X_i}(\theta) \geq u\theta - \tfrac12 K\theta^2 \,.
\end{equation}
Using this bound with $\theta = u/K$ in \eqref{eqn:3} and assuming $u$ is small enough, the result follows with $C=(2K)^{-1}$ in \eqref{eqn:2}.
\end{proof}

The next lemma shows that the empirical distribution $\ol{Z}_n$ concentrates around its mean $\nu_n$. For what follows, given $\zeta \in \cZ$,
we shall formally write $\bbP(\cdot | Z_0 = \zeta)$ to denote the probability measure under which $Z=(Z_n)_{n \geq 0}$ is defined exactly as under $\bbP$, only that $Z_0 = \zeta$. The corresponding expectation will be denoted by $\bbE(\cdot | Z_0 = \zeta)$.

\begin{lem}
\label{l:13}
There exists $C, C^\prime >0$ such that for all $u > 0$ sufficiently small, finite point measure $\zeta$, subset $A \subseteq \bbR$ and $n \geq 1$,
\begin{equation}
\label{eqn:9}
  \bbP \Big(\ol{Z}_n(A) > 
    \big \la \ol{\zeta} \,, \nu_n(A - \cdot) \big\ra  + u\,\Big|\, Z_0 = \zeta \Big) \leq C e^{-C^\prime u^2 |\zeta|} \,.
\end{equation}
The same holds if we replace $>$ with $<$ and $+u$ with $-u$.
\end{lem}

\begin{proof}
For $x \in \zeta$, let $Z_n^{(x)}$ be distributed as $Z_n(\cdot - x)$, where $Z_n$ is the branching random walk process with $Z_0 = \delta_0$. Suppose also that $\big( Z_n^{(x)} :\: x \in \zeta \big)$ are all defined on the same probability space and are independent of each other. Writing $\wh{Z}_n^{(x)} := Z_n^{(x)}/\bbE \big|Z_n^{(x)}\big|$, under $\bbP \big( \cdot | Z_0 = \zeta \big)$ we have
\begin{equation}
  \ol{Z}_n(A) \overset{\rmd} = 
  \frac{\tfrac{1}{|\zeta|} \sum_{x \in \zeta} \wh{Z}_n^{(x)}(A)}
       {\tfrac{1}{|\zeta|} \sum_{x \in \zeta} \big|\wh{Z}_n^{(x)}\big| } \, .
\end{equation}
Therefore, the left hand side of \eqref{eqn:9} is bounded above by
\begin{equation}
\label{eqn:12}
    \bbP \Big(\tfrac{1}{|\zeta|} \sum_{x \in \zeta} \big|\wh{Z}_n^{(x)}\big|  < 
      1-\tfrac{u}{2} \Big) \, + \, 
     \bbP \Big(\tfrac{1}{|\zeta|} \sum_{x \in \zeta} \wh{Z}_n^{(x)}(A)  > 
       \tfrac{1}{|\zeta|} \sum_{x \in \zeta} \nu_n(A - x) + \tfrac{u}{3} \Big)
\end{equation}
as long as $u$ is small enough. Now Theorem~4 in \cite{ATH1} gives a uniform bound
on the moment generating function $e^{\theta |\wh{Z}^{(x)}_n|}$ for all $n \geq 1$, $x \in \zeta$ and $\theta \in [0,\theta_0]$, for some $\theta_0 > 0$. This uniform bound can be extended to include
also the moment generating functions of (the stochastically smaller) $\wh{Z}_n^{(x)}(A)$ for all $A \subseteq \bbR$ and $x \in \zeta$ in the same range of $\theta$. The non-negativity of all these random variables imply that we may extend this bound also to all $\theta < 0$. Thus, it is not difficult to see that the family of random variables 
\begin{equation}
  \bbX = \{\pm(\wh{Z}^{(x)}_n(A) - \nu_n(A)) : \: n \geq 1,\, x \in \zeta,\, A \subseteq \bbR\}
\end{equation}
satisfies the conditions in Lemma~\ref{l:12}, whence \eqref{eqn:12}
is bounded above by $C e^{-C^\prime u^2 |\zeta|}$ for some $C, C^\prime > 0$ as desired. 

Replacing $A$ with $A^\rmc$ , we obtain \eqref{eqn:9} 
with $<$, $-u$ in place of $>$, $+u$. 
\end{proof} 

\section{Proof of Theorem~\ref{thm:1}}
\label{s:4}
We are now ready for,
\begin{proof}[Proof of Theorem~\ref{thm:1}]
Let $A \in \cA$ and $p \in [0,1)$ be given such that $p > \nu(A)$. Suppose first that $I(A,p) < \infty$ and that $I(A,\cdot)$ is continuous at $p$. We shall show~\eqref{e:5} by proving that the right hand side bounds the left hand side from below and from above. Starting with the upper bound, by the obvious monotonicity of $b_0 \wt{I}^\pm(A,\cdot)/\alpha^\pm$ if $I(A,\cdot)$ is continuous at $p$, it must follow that either $b_0 \wt{I}^+(A,p) / \alpha^+ = I(A,p)$ and $\wt{I}^+(A,\cdot)$ is continuous at $p$ or $b_0 \wt{I}^-(A,p) / \alpha^- = I(A,p)$ and $\wt{I}^-(A,\cdot)$ is continuous in $p$. Since both cases are handled exactly in the same way, let us assume that the former happens, namely that $b_0 \wt{I}^+(A,p) / \alpha^+ = I(A,p)$ and that $p$ is a continuity point for $\wt{I}^+(A, \cdot)$. This implies that for any $\epsilon > 0$, there exists $\delta >0$ and $0 < x < \wt{I}^+(A,p) + \epsilon$ such that, $\nu(A - x) \geq p + \delta$

Now set $x_n := \sqrt{n} x$, $m_n := \lfloor x_n/\alpha^+ \rfloor$ and use Lemma~\ref{l:3.2} to conclude that if $n$ is large enough, there exists $\zeta_{m_n} \in \cZ_{m_n}$ such that
\begin{equation}
\label{e:105}
\ol{\zeta}_{m_n} \Big(\big[x_n \big(1-\tfrac{2\epsilon}{\alpha^+}) \,,\, 
						   x_n \big(1+ \tfrac{2\epsilon}{\alpha^+}\big) \big]\Big) \geq 1-\epsilon
\quad , \qquad
\bbP(Z_{m_n} = \zeta_{m_n}) \geq \rme^{-\rme^{(b_0 + \epsilon) m_n}} \,.
\end{equation}
Using the Markov property, we can then write
\begin{equation}
\label{e:106}
\bbP \big( \ol{Z}_n (\sqrt{n}A) \geq p \big)
\geq \bbP \big( Z_{m_n} = \zeta_{m_n} \big)
	\bbP \big( \ol{Z}_{n-m_n} (\sqrt{n}A) \geq p \,\big|\, Z_0 = \zeta_{m_n} \big) \,.
\end{equation}
The first probability on the right hand side of the inequality is bounded below by
\begin{equation}
\exp\bigg\{-\rme^{(b_0 + \epsilon) \sqrt{n} (\wt{I}^{+}(A,p) + \epsilon)/\alpha^+}\bigg\} 
\geq \exp\bigg\{-\rme^{\sqrt{n} (I(A,p) + C \epsilon)}\bigg\} \,,
\end{equation}
provided $\epsilon$ is small enough.

For the second probability in the right hand side of~\eqref{e:106}, recall the definition of $\nu_n$ in~\eqref{e:3.6} and use~\eqref{e:105} and Lemma~\ref{l:11} with $n$ large enough to bound $\big \la \ol{\zeta}_{m_n} \,,\, \nu_{n-m_n} \big(\sqrt{n}A - \cdot\big) \big \ra$ from below by
\begin{equation}
\nonumber
\begin{split}
\Big \la \ol{\zeta}_{m_n} & \,, \nu_{n-m_n} \Big(\sqrt{n-m_n} \Big(\sqrt{\tfrac{n}{n-m_n}} A - \tfrac{\cdot}{\sqrt{n-m_n}} \Big) \Big) 1_{[x_n(1-2\epsilon/\alpha^+)\,,\, x_n(1+2\epsilon/\alpha^+)]} \Big \ra  \\
& \geq 
(1-\epsilon) \inf \Big\{ \nu_{n-m_n} \Big(\sqrt{n-m_n} \Big(\sqrt{\tfrac{n}{n-m_n}} A - z \Big)\Big) :\: z \in [x(1-2\epsilon/\alpha^+) \,,\, x(1+3\epsilon/\alpha^+)] \Big\} \\
& \geq  
	\inf \Big\{ \nu \Big(\sqrt{\tfrac{n}{n-m_n}} A - z \Big) :\: z \in [x(1-2\epsilon/\alpha^+) \,,\, x(1+3\epsilon/\alpha^+)] \Big\} - 2 \epsilon\\
\end{split}
\end{equation}
Then, thanks to part 1 of Lemma~\ref{l:10} and the definition of $x$, by choosing $\epsilon$ small enough, we can have the last term bounded from below by $p + \delta / 2$ for all $n$ large enough.
It then follows from Lemma~\ref{l:13} that for all $n$ large enough, the second probability in the right hand side of~\eqref{e:106} is at least 
\begin{equation}
\bbP \Big( \ol{Z}_{n-m_n} (\sqrt{n}A) \geq \big \la \ol{\zeta}_{m_n} \,,\, \nu_{n-m_n} \big(\sqrt{n}A - \cdot \big) \big \ra - \tfrac{\delta}{2} \,\Big|\, Z_0 = \zeta_{m_n} \Big) 
\geq 1 - C \rme^{-C' \delta^2 |\zeta_{m_n}|} > 1/2\,.
\end{equation}
Using the above bounds in~\eqref{e:106}, we then get
\begin{equation}
\limsup_{n \to \infty} \tfrac{1}{\sqrt{n}} \log \big[-\log \bbP \big( \ol{Z}_n (\sqrt{n}A) \geq p \big)\big] \leq I(A,p) + C \epsilon \,,
\end{equation}
and since $\epsilon$ are arbitrary this gives the upper bound.

Turning to the lower bound, let $\epsilon' > 0$, set $m_n := b_0^{-1} I(A,p) \sqrt{n} (1-\epsilon')$ and use the Markov property to write
\begin{equation}
\label{e:5.6}
\bbP \big(\ol{Z}_n(\sqrt{n}A) \geq p\big) = \sum_{\zeta_{m_n} \in \cZ_{m_n}}
\bbP \big(Z_{m_n} = \zeta_{m_n} \big) \bbP \big(\ol{Z}_{n-m_n}(\sqrt{n}A) \geq p  \,\big|\, 
Z_0 = \zeta_{m_n} \big) \,.
\end{equation}
Then for any $\delta > 0$, using Lemma~\eqref{l:8} 
 with $\epsilon := \min\{\epsilon' \alpha^+/2,\, \epsilon' \alpha^-/2,\, \delta\}$ for all $n$ large enough, we have
\begin{equation}
\ol{\zeta}_{m_n} \Big(\sqrt{n} (1-\epsilon'/2) \big[\wt{I}^-(A,p) \,,\,\wt{I}^+(A,p)\big]^\rmc \Big)
\leq \ol{\zeta}_{m_n} \Big(m_n \big[-\alpha^- - \epsilon \,,\, \alpha^+ + \epsilon\big]^\rmc \Big) \leq \delta \,.
\end{equation}

Using Lemma~\ref{l:10}, Lemma~\ref{l:11}, for all $\delta > 0$ we can find $n$ large enough, such that
\begin{equation}
\begin{split}
\label{e:5.8}
\big \la \ol{\zeta}_{m_n} & \,,\, \nu_{n-m_n} \big(\sqrt{n}A - \cdot\big) \big \ra  \\
& \leq \Big \la \ol{\zeta}_{m_n} \,,\, \nu_{n-m_n} \Big(\sqrt{n-m_n} \Big(\sqrt{\tfrac{n}{n-m_n}} A - \tfrac{\cdot}{\sqrt{n-m_n}} \Big)\Big) 1_{\{
 \sqrt{n} (1-\epsilon'/2) [\wt{I}^-(A,p) \,,\,\wt{I}^+(A,p)]\}}  \Big \ra 
 + \delta   \\
& \leq \Big \la \ol{\zeta}_{m_n} \,,\, \nu \Big(\sqrt{\tfrac{n}{n-m_n}} A - \tfrac{\cdot}{\sqrt{n-m_n}} \Big)  1_{\{
 \sqrt{n} (1-\epsilon'/2) [\wt{I}^-(A,p) \,,\,\wt{I}^+(A,p)]\}} \Big \ra + 2 \delta \\
& \leq \Big \la \ol{\zeta}_{m_n} \,,\, \nu \Big(A - \tfrac{\cdot}{\sqrt{n}} \Big)  1_{\{
 \sqrt{n} (1-\epsilon'/2) [\wt{I}^-(A,p) \,,\,\wt{I}^+(A,p)]\}} \Big \ra + 3 \delta \,.
\end{split}
\end{equation}

By definition, the integral in the last line is at most 
\begin{equation}
\max \big\{ \nu \big( A- y \big) :\: y \in  \big[-(1-\epsilon'/2)(\wt{I}^-(A,p) \,,\, (1-\epsilon'/2)\wt{I}^+(A,p) \big] \big\} < p \,.
\end{equation}
Therefore, if we choose $\delta > 0$ small enough, we can ensure that 
\begin{equation}
\big \la \ol{\zeta}_{m_n} \,, \nu_{n-m_n} \big(\sqrt{n}A - \cdot\big) \big \ra 
\leq p - u \,.
\end{equation}
for some $u > 0$ and all $n$ large enough. This implies via Lemma~\ref{l:13} that
\begin{equation}
\bbP \big(\ol{Z}_{n-m_n}(\sqrt{n}A) \geq p  \,\big|\, 
Z_0 = \zeta_{m_n} \big) 
\leq C \rme^{-C' u^2 |\zeta_{m_n}|}
\leq C \exp \big(\!-\! C' u^2 \rme^{I(A,p)(1-\epsilon') \sqrt{n}} \big)\,,
\end{equation}
where the last inequality follows since $|Z_{m_n}| \geq \rme^{b_0 m_n}$.
Plugging this in~\eqref{e:5.6}, we obtain
\begin{equation}
\liminf_{n \to \infty} \tfrac{1}{\sqrt{n}}
	\log \big[-\log \bbP \big( \ol{Z}_n (\sqrt{n}A) \geq p \big)
	\geq I(A,p) (1- \epsilon') \,.
\end{equation}
Since $\epsilon' > 0$ was arbitrary, the lower bound follows.

Suppose now that $I(A,p) = \infty$ and that $J(A,\cdot)$ is continuous at $p$. By Lemma~\ref{l:10} part 4, we know that $J(A,p) > 0$. Again, we shall show~\eqref{e:6} by bounding the left hand side by the right hand side separately from below and from above. Starting with the upper bound, fix any $\epsilon' > 0$ and use the continuity of $\wt{J}(A,\cdot) = b_0^{-1} J(A,\cdot)$ at $p$ and Lemma~\ref{l:10} part 3, to conclude that there exists $\delta > 0$, $u \in (0,1)$ and $x \in \bbR$ such that
\begin{equation}
\label{e:122}
0 < u < \wt{J}(A,p) + \epsilon'
\quad, \qquad 
\nu\big(\big(A/\sqrt{1-u}\big) - x\big) \geq p + \delta \,.
\end{equation}
Then with $m_n = \lceil un \rceil$, any $\epsilon > 0$ and all $n \geq 1$ large enough, use Lemma~\ref{l:11a} to find $\zeta_{m_n} \in \cZ_{m_n}$ such that
\begin{equation}
\label{e:124}
\ol{\zeta}_{m_n} \big([-\epsilon m_n, \epsilon m_n]\big) \geq 1 - \epsilon
\quad, \qquad
|\zeta_{m_n}| \leq \rme^{m_n(b_0 + \epsilon)}
\quad, \qquad
\bbP(Z_{m_n} = \zeta_{m_n}) \geq \rme^{-\rme^{m_n(b_0 + \epsilon)}} \,.
\end{equation}

Now set 
\begin{equation}
x_n := x \sqrt{n-m_n} \quad, \qquad
k_n := \lceil m_n \sqrt{\epsilon} \rceil \quad,\qquad
m'_n = m_n + k_n  \,,
\end{equation}
and define $\zeta_{m'_n} \in \cZ_{m'_n}$ from $\zeta_{m_n}$, by having each particle $y \in \zeta_{m_n}$ with $|y| \leq \epsilon m_n$ reproduce according to some $\zeta^{(y)}_{k_n} \in \cZ_{k_n}$ satisfying
\begin{equation}
\label{e:125}
\ol{\zeta}^{(y)}_{k_n} \Big(\big[x_n - y - k_n^{1/3} \,, x_n - y + k_n^{1/3} \big] \Big)
\geq 1-\epsilon 
\quad, \quad
\big|\zeta^{(y)}_{k_n} \big| \leq \rme^{b k_n}
\quad, \quad
\bbP \big(Z_{k_n} = \zeta^{(y)}_{k_n} \big) 
\geq \rme^{-\rme^{2b k_n}} \,,
\end{equation}  
while all other particles reproduce according to $\zeta_{k_n} := (\zeta_1)^{*k_n}$, where $\zeta_1 \in \cZ_1(b_0)$ is any fixed but arbitrary configuration. The existence of such $\zeta^{(y)}_{k_n}$ is guaranteed by Lemma~\ref{l:10a}, since 
\begin{equation}
\frac{|x_n - y|}{k_n} \leq \frac{|x| \sqrt{n-m_n} + \epsilon m_n}{m_n\sqrt{\epsilon}} 
\leq \frac{|x| u^{-1/2} \sqrt{m_n} + \epsilon m_n}{m_n\sqrt{\epsilon}} 
\leq 2 \sqrt{\epsilon}  < r  \,,
\end{equation}
for $\epsilon$ small enough and then all $n$ large enough (above $r$ and $b$ are as in Lemma~\ref{l:10a} for $\eta = 1/3$). Then, by~\eqref{e:124} and~\eqref{e:125}, the quantity $\ol{\zeta}_{m'_n} \big( \big[x_n - k_n^{1/3} ,\, x_n + k_n^{1/3} \big]^\rmc \big)$ is at most  
\begin{equation}
\label{e:127}
\frac{ \epsilon |\zeta_{m_n}| \rme^{b_0 k_n} + \big \la \zeta_{m_n} \,,\, \epsilon \big|\zeta^{(\cdot)}_{k_n} \big| 
	1_{[-\epsilon m_n \,,\, \epsilon m_n]} \big \ra}{|\zeta_{m'_n}|} 
	\leq \frac{\epsilon |\zeta_{m_n}| \rme^{b_0 k_n}}{|\zeta_{m_n}| \rme^{b_0 k_n}}
	+ \frac{\epsilon |\zeta_{m'_n}|}{|\zeta_{m'_n}|} \,,
\end{equation}
which is smaller than $2\epsilon$. Moreover, for all $n$ large enough,
\begin{equation}
\nonumber
\begin{split}
\bbP \big( Z_{m_n'} = \zeta_{m_n'} \big) 
& = \bbP \big(Z_{m_n} = \zeta_{m_n}\big) 
\Big( \min \big\{\rme^{-\rme^{2bk_n}}, \rme^{-C k_n \rme^{k_n b_0}}\big\} \Big)^{|\zeta_{m_n}|} \\
& \geq \exp \big(\!-\!\rme^{m_n(b_0 + \epsilon)} - \rme^{m_n (b_0 + \epsilon) + 2b k_n} \big)
\geq \rme^{-\rme^{-(b_0+3b\sqrt{\epsilon}) m_n}} 
\geq \rme^{-\rme^{-(J(A,p) + C(\sqrt{\epsilon} + \epsilon'))n}} \,,
\end{split}
\end{equation}
where we have used the definition of $m_n$, the bound on $u$ in~\eqref{e:122} and the fact that $J(A,p) > 0$.

We now write,
\begin{equation}
\label{e:129}
\bbP \big( \ol{Z}_n (\sqrt{n} A) \geq p \big) 
\geq \bbP \big( Z_{m'_n} = \zeta_{m'_n} \big) 
\bbP \big( \ol{Z}_{n-m'_n} (\sqrt{n} A) \geq p \, \big|\, Z_0 = \zeta_{m'_n} \big) \,.
\end{equation}
The first probability on the right hand side, we have just bounded from below. As for the second, by~\eqref{e:127}, Lemma~\ref{l:10} and Lemma~\ref{l:11}, 
\begin{equation}
\begin{split}
\big \la \ol{\zeta}_{m_n'} ,\, 
	\nu_{n-m_n} \big(\sqrt{n}A -\cdot \big) \ra 
& \geq (1-2\epsilon) 
\inf \Big\{ \nu_{n-m_n} \big(\sqrt{n}A -y \big) :\: |y - x_n| \leq k_n^{1/3} \Big\} \\
& \geq 
\inf \Big\{ \nu \big(\sqrt{\tfrac{n}{n-m_n}}A - y \big) :\: |y - x| \leq C n^{-1/6} \Big\} - 3 \epsilon \\
& \geq \nu \big( \big(A/\sqrt{1-u}\big) - x \big) - 4\epsilon 
\, \geq \, p + \delta - 4\epsilon \,.
\end{split}
\end{equation}
The latter can be made at least $p+\delta/2$ by choosing $\epsilon$ small enough. But then, by Lemma~\ref{l:13}, the second probability on the right hand side of~\eqref{e:129} is at least
\begin{equation}
1-C \rme^{-C'|\zeta_{m'_n}| \delta^2 } \geq 1/2 \,,
\end{equation}
for $n$ large enough, since $|\zeta_{m'_n}| \geq \rme^{b_0 m'_n}$.
Combining the two lower bounds and~\eqref{e:129} we obtain
\begin{equation}
\limsup_{n \to \infty} n^{-1} \log \big[-\log \bbP \big( \ol{Z}_n (\sqrt{n} A) \geq p \big) \big]
\leq J(A,p) + C (\sqrt{\epsilon} +  \epsilon') \,.
\end{equation}
Since $\epsilon, \epsilon'$ can be made arbitrarily small, the upper bound follows.

Finally, we turn to the lower bound in the case $I(A,p) = \infty$. Fix, again, any $\epsilon > 0$ and set 
$u := \wt{J}(A,p)(1-\epsilon)$ and $m_n := \lfloor u n \rfloor$. By definition of $\wt{J}(A,p)$, for $\epsilon$ as above,  there exists $\delta > 0$ such that $\nu(A/\sqrt{1-u} - x) \leq p - \delta$ for all $x \in \bbR$. Using Lemma~\ref{l:10} and Lemma~\ref{l:11}, it follows that for $n$ large enough and any $\zeta_{m_n} \in \cZ_{m_n}$
\begin{equation}
\big \la \ol{\zeta}_{m_n} \,,\, \nu_{n-m_n} (\sqrt{n}A - \cdot) \big \ra 
\leq \sup_{x \in \bbR} \nu_{n-m_n} \big(\sqrt{n}A - x \big) \leq p - \delta / 2 \,.
\end{equation}
But then, using~\eqref{e:5.6} as in the case of $I(A,p) < \infty$, Lemma~\ref{l:13} and the fact that $|\zeta_{m_n}| \geq \rme^{b_0 m_n}$, we get for all $n$ large enough,
\begin{equation}
\bbP \big(\ol{Z}_n(\sqrt{n}A) \geq p\big) 
\leq \rme^{-C \delta^2 \rme^{b_0 m_n}}
\leq \exp \big(\!-\!\rme^{J(A,p) (1- 2\epsilon) n} \big) \,.
\end{equation}
It then follows that
\begin{equation}
\liminf_{n \to \infty} \tfrac{1}{n}
	\log \big[-\log \bbP \big( \ol{Z}_n (\sqrt{n}A) \geq p \big)
	\geq J(A,p)(1 - 2 \epsilon) \,.
\end{equation}
As $\epsilon$ can be chosen arbitrarily small, this completes the lower bound in the second case and finishes the proof. 
\end{proof}

\section{Proof of Proposition~\ref{p:1.2}}
\label{s:5}
In this section we prove Proposition~\ref{p:1.2}. We shall need the following standard fact, whose short proof is included for the sake of completeness. For what follows, let $(S_n)_{n \geq 0}$ be a simple symmetric random walks with $\pm 1$ steps, starting from $0$. Then,
\begin{lem}
With $\beta$ as in~\eqref{e:1.17} and any $x \in [0,1]$, we have
\begin{equation}
\label{e:1.20}
\lim_{n \to \infty} n^{-1} \log \bbP \big(S_k \geq x k \,,\,\, k=0, \dots, n \big) = -\beta(x) \,,  
\end{equation}
\end{lem}
\begin{proof}
If $x = 1$, then the probability above is clearly equal to $2^{-n}$ in agreement with $\beta(1) = \log 2$. We therefore suppose for the rest of the argument that $x \in [0,1)$ and provide matching upper and lower bounds separately. For the upper bound, observe that the probability in question is bounded from above by $\bbP(S_n \geq xn)$. Then it follows from Cram\'{e}rs theorem (or just the exponential Chebyshev's inequality) with respect to the distribution $\frac12\delta_{-1} + \frac12\delta_{+1}$, that the left hand side (with the limit replaced by limit superior) is at most $-\beta(x)$.

For a lower bound, we can first intersect the event in~\eqref{e:1.20} with the event $\{S_n \leq xn + n^{2/3}\}$. Then, using a standard tilting argument, for any $\theta \in \bbR$ the probability in~\eqref{e:1.20} can be bounded from below by,
\begin{equation}
\exp \big(\!-n\big(\theta x - L(\theta)\big) - \theta n^{2/3} \big) \bbP^\theta \big(S_k \geq x k \,,\,\, k=0, \dots, n 
\,,\,\, S_n \leq xn + n^{2/3} \big) \,,
\end{equation}
where $\bbP^\theta$ is a probability measure, under which $S_n$ is the sum of $n$ i.i.d. random variables with distribution $\big(\rme^{-\theta}/(\rme^{\theta} + \rme^{-\theta})\big) \delta_{-1} + \big(\rme^{\theta}/(\rme^{\theta} + \rme^{-\theta})\big) \delta_{+1}$ and $L(\theta) := \log \bbE \rme^{\theta S_1}$.
Choosing $\theta = \tanh^{-1}(x) \in [0, \infty)$, gives $\bbE^\theta S_1 = L'(\theta) = x$, where $\bbE^\theta$ denotes expectation with respect to $\bbP^\theta$. Moreover, in this case $\theta x - L(\theta) = \beta(x)$. Setting also $\wh{S}_k := S_k - xk$, with this choice of $\theta$ the last display can be written as
\begin{equation}
\rme^{-n \beta(x) - \theta n^{2/3}} \bbP^\theta \big(\wh{S}_k \geq 0 \,,\,\, k=0, \dots, n 
\,,\,\, \wh{S}_n \leq n^{2/3} \big) \,.
\end{equation}
Since $(\wh{S}_k :\: k \geq 0)$ is a random walk with mean $0$ and finite non zero variance, under $\bbP^\theta$, the last probability is at least $C n^{-1/2}$ for some $C = C(\theta) > 0$ (see, e.g., Theorem 5.1.7 in~\cite{lawler2010random}). This shows the lower bound.
\end{proof}

\begin{proof}[Proof Proposition~\ref{p:1.2}]
If $b \in \{b_0, b_1\}$ then by definition of the law of $Z_1$, for all $n \geq 1$ the set $\cZ_n(b)$ consists of a single particle measure $\zeta$, in which at every generation, all particles reproduce according to $\rme^{b}\big(\tfrac12 \delta_{-1} + \tfrac12 \delta_{+1}\big)$. It follows that $\la \ol{\zeta} \,,\, x \ra = 0$ in agreement with both statements of the proposition. We shall therefore suppose for the rest of the proof that $b \in (b_0, b_1)$. Given a particle in $Z_n$, we shall refer to the sequence of heights of its ancestors from generation $0$ to generation $n$ as the ``path'' of the particle. Clearly, under the assumptions of Example~\ref{sss:3}, the set of possible paths of particles in generation $n$  identifies with all $2^n$ possible trajectories of $S_n$ as defined above.

Starting with the first statement of the proposition, first observe that $\gamma(b) = 1$ whenever $b_1 - b \geq \log 2$ and hence~\eqref{e:1.18} trivially holds in this case. Assuming therefore that $b_1 - b < \log 2$, pick any $\epsilon > 0$ small enough and use~\eqref{e:1.20} to bound the number of simple random walk paths which lead to $x \geq (\beta^{-1}(b_1 - b) + \epsilon)n$ by
\begin{equation}
2^n \bbP \big(S_n \geq (\beta^{-1}(b_1 - b) + \epsilon)n \big)
\leq \exp \Big(n \big(\log 2 - \beta\big(\beta^{-1}(b_1 - b) + \epsilon\big) + 2\delta \Big)
\leq \rme^{n (\log 2 + b - b_1 - \delta)} \,,
\end{equation}
for some $\delta > 0$ provided $n$ is large enough, where we have used the strict monotonicity of $\beta$.

Since each path can be taken by at most $(\rme^{b_1}/2)^n$, any $\zeta_n \in \cZ_n$ with $n$ large enough will satisfy
\begin{equation}
\zeta_n \big(\big[(\beta^{-1}(b_1 - b) + \epsilon)n \,,\, \infty\big)\big) \leq 
\exp \big(n (\log 2 + b - b_1 - \delta + b_1 - \log 2) \big) \leq \rme^{n(b - \delta)} \,,
\end{equation}
In particular, for any $\zeta_n \in \cZ_n(b)$ with $n$ large enough,
\begin{equation}
\label{e:5.4}
\begin{split}
\la \ol{\zeta}_n,\, x \ra & \leq n\big(\beta^{-1}(b_1 - b) + \epsilon \big)\, +\, n \ol{\zeta}_n \big(\big[n(\beta^{-1}(b_1 - b) + \epsilon) \,,\, \infty\big)\big) \\
& \leq n \big(\beta^{-1}(b_1 - b) +\epsilon \big) + n \rme^{-\delta n}
\leq n \big(\beta^{-1}(b_1 - b) + 2\epsilon \big) = n (\gamma(b) + 2\epsilon) \,,
\end{split}
\end{equation}
where we have used that the support of $\zeta_n$ is contained in $[-n, n]$. 

Now from any $\zeta \in \cZ_{m}(b)$ for some $m \geq 1$, we can construct $\zeta_n \in \cZ_n(b)$ with $n=mk$, by letting $\zeta_n := \zeta^{*k}$. Then for any $\epsilon > 0$ by choosing $k$ large enough, we can deduce from~\eqref{e:5.4} that,
\begin{equation}
m^{-1} \la \ol{\zeta} \,,\, x \ra = 
n^{-1} k \la \ol{\zeta} \,,\, x \ra =
n^{-1} \la \ol{\zeta}_n \,,\, x \ra \leq \gamma(b) + 2\epsilon \,,
\end{equation}
which shows~\eqref{e:1.18} must hold for such $\zeta$ and completes the proof of the first statement of the proposition.

Turning now to the second statement, assume first that $b_1 - b \leq \log 2$ and for $n \geq 1$ let $\zeta_n$ be the configuration of particles obtained by having all particles in generation $0 \leq k < n$, branch by $\rme^{b_1}$ if they are at height $x \geq k \beta^{-1}(b_1 - b)$ and otherwise by $\rme^{b_0}$. Letting $b' \in [b_0, b_1]$ be such that $\zeta_n \in \cZ_n(b')$, it is enough to show that for any $\epsilon > 0$ if $n$ is large enough then
\begin{equation}
\label{e:1.21}
b' \in [b-\epsilon, b+\epsilon]
\quad, \qquad
n^{-1} \la \ol{\zeta}_n\,, x \ra \geq \beta^{-1}(b_1 - b) - \epsilon \,.
\end{equation}
To show~\eqref{e:1.21}, first use~\eqref{e:1.20} to conclude that with any $\epsilon > 0$ as long as $n$ is large enough,
\begin{equation}
\label{e:1.23}
\begin{split}
\zeta_n \big([n\beta^{-1}(b_1-b), \infty)\big) & \leq
2^n \bbP \big(S_n \geq n\beta^{-1}(b_1 - b)\big) \big(\rme^{b_1}/2\big)^n \\
& \leq \exp \Big(n\big(\log 2 - \beta\big(\beta^{-1}(b_1-b)\big) + \epsilon + b_1 - \log 2\big)\Big)
\leq \rme^{n(b+\epsilon)} \,,
\end{split}
\end{equation}
and
\begin{equation}
\label{e:1.24}
\begin{split}
\zeta_n \big([n\beta^{-1}(b_1-b), \infty)\big) & \geq
2^n \bbP \big(S_k \geq k\beta^{-1}(b_1 - b) \,,\,\, k=0, \dots, n\big) \big(\rme^{b_1}/2\big)^n \\
& \geq \exp \Big(n\big(\log 2 - \beta\big(\beta^{-1}(b_1-b)\big) - \epsilon + b_1 - \log 2\big)\Big)
\geq \rme^{n(b-\epsilon)} \,.
\end{split}
\end{equation}

At the same time, by considering the last generation $k$ when the path of particle was at height $x \geq k \beta^{-1}(b_1-b)$, we can write for any $y \geq 0$,
\begin{equation}
\label{e:1.25}
\begin{split}
\zeta_n\big(\big(\!-\!\infty \,,\,\, n\beta^{-1}(b_1 - b) - y\big] \big) & \leq 
\sum_{k=0}^{n-\lfloor y/2 \rfloor} 2^n \bbP \big(S_k \geq k \beta^{-1}(b_1-b) \big) \big(\rme^{b_1}/2\big)^{k+1}
\big(\rme^{b_0}/2\big)^{n-k} \\
& \leq \rme^{b_1}/2 \sum_{k=0}^{n-\lfloor y/2 \rfloor} \exp \Big(k\big(-\beta\big(\beta^{-1}(b_1 - b)\big) + \epsilon + b_1 \big)
+ (n-k) b_0 \Big)  \\
& = \rme^{n(b+\epsilon)+b_1}/2 \, \sum_{k=0}^{n-\lfloor y/2 \rfloor } \rme^{-(n-k) (b - b_0)}
\leq \rme^{n(b+2\epsilon) - y(b-b_0)/2} \,.
\end{split}
\end{equation}
Above we have used that $b > b_0$ and that $n-k \geq y/2$ for the path to reach height $x \leq n\beta^{1}(b_1 - b) - y$ at time $n$. As before, $\epsilon > 0$ can be arbitrarily small as long as $n$ as large enough.

Now, the upper bounds in~\eqref{e:1.23} and in~\eqref{e:1.25} with $y=0$, show that $|\zeta_n| \leq \rme^{n(b+3 \epsilon)}$, while the lower bound in~\eqref{e:1.24} yields $|\zeta_n| \geq \rme^{n(b - \epsilon)}$. Altogether this shows the first part of~\eqref{e:1.21} with $3\epsilon$ in place of $\epsilon$. Using the lower bound on $|\zeta_n|$ together with~\eqref{e:1.25} then gives for all $y \geq 0$,
\begin{equation}
\ol{\zeta}_n \big(\big(\!-\!\infty \,,\, n\beta^{-1}(b_1 - b)-y \big]\big) 
\leq \rme^{3\epsilon n - y(b-b_0)/2} \,.
\end{equation} 
Plugging in $y=\epsilon' n$ with $\epsilon' = 8\epsilon/(b-b_0)$ and using that the bound on the support of $\zeta_n$, we get
\begin{equation}
\nonumber
\begin{split}
\la \ol{\zeta}_n \,,\, x \ra
& \geq \big(n \beta^{-1}\big(b_1 - b) - y\big) \ol{\zeta}_n 
\big( \big(n \beta^{-1}\big(b_1 - b) - y \,,\, \infty \big) \big)
- n \ol{\zeta}_n 
\big(\big(\!-\!\infty ,\, n \beta^{-1}\big(b_1 - b) - y \big] \big) \\
& \geq n \big(\beta^{-1}\big(b_1 - b) - \epsilon'\big) \big( 1 - \rme^{-\epsilon n} \big)
- n  \rme^{-\epsilon n} \geq n \big(\beta^{-1}\big(b_1 - b) - 2\epsilon'\big) \,.
\end{split}
\end{equation}
for all $n$ large enough, which shows the second part of~\eqref{e:1.21} with $2\epsilon'$ in place of $\epsilon$ and completes the proof for the second statement in the case $b_1 -b \leq \log 2$.

When $b_1 - b > \log 2$, we first find $\theta \in [0,1)$ such that $\theta \rme^{b_0} + (1-\theta) \rme^{b_1} = \rme^{b + \log 2}$. This is always possible since $b_0 < b + \log 2 < b_1$. 
Then for any $\epsilon' > 0$, if $M$ is large enough 
\begin{equation}
\label{e:1.29}
\rme^{b + \log 2} \leq 
\frac{\lceil M\theta \rceil \rme^{b_0} + (M - \lceil M\theta \rceil) \rme^{b_1}}{M} \leq \rme^{b + \log 2 + \epsilon'}  \,.
\end{equation}
This shows that as soon as there are $M$ particles, we can have them multiply in a way which is equivalent to letting each of them reproduce according to the (non-point) measure $\rme^{b'_1} \big(\tfrac12 \delta_- + \tfrac12 \delta_+ \big)$ for some $b'_1 \in [b+\log 2 \,,\, b+\log 2 + \epsilon']$. 

For any $n' \geq 1$, we now construct $\zeta'_{n'}$ in a similar way to the one used to construct $\zeta_n$ in the proof of the lower bound above, albeit with two notable differences. First, we start with $M$ particles at height $0$, where $M$ is large enough so that~\eqref{e:1.29} holds. Second, instead of reproducing by $\rme^{b_1}$, a particle in generation $0 \leq k < n$ now reproduces according to~\eqref{e:1.29} whenever its height $x = k\beta^{-1}(\log 2) = k$ and otherwise according to $\rme^{b_0}$. It is not difficult to see that the proof of~\eqref{e:1.21}
can be used almost without a change to show that $\zeta'_{n'}/M \in \cZ_{n'}(b')$ with
\begin{equation}
\label{e:1.30}
b' \in [b-\epsilon, b+\epsilon]
\quad, \qquad
n'^{-1} \la \ol{\zeta'}_n\,, x \ra \geq \beta^{-1}(\log 2) - \epsilon = 1 - \epsilon \,,
\end{equation}
for arbitrarily small $\epsilon > 0$ as long as $n'$ is large enough and provided we choose $M$ as large as needed.

Next, we construct $\zeta_{n'+2m}$ using $\zeta'_{n'}$ by having all particles reproduce at rate $\rme^{b_0}$ for the first $2m$ generations and then letting all particles at height $0$ reproduce according to $\zeta'_{n'}$ for the next $n'$ generations, while the rest keep reproducing by $\rme^{b_0}$. Clearly, by choosing $m$ large enough, we can make the number of particles at at height $0$ in generation $2m$ be equal to some $M$ for which~\eqref{e:1.30} holds with any given $\epsilon > 0$ and all $n'$ large enough.

Now, setting $n = n'+2m$, we have for all $n$ large 
\begin{equation}
\rme^{n(b-2\epsilon)} \leq 
\rme^{n'(b-\epsilon)} \leq |\zeta'_{n'}| \leq \, |\zeta_n| \,
\leq |\zeta'_{n'}| + \rme^{n b_0} \leq 
M \rme^{n'(b + \epsilon)} + \rme^{n b_0} \leq
\rme^{n(b + 2\epsilon)} \,,
\end{equation}
and
\begin{equation}
\la \ol{\zeta}_n \,,\, x \ra \geq 
\la \ol{\zeta}'_{n'} \,,\, x \ra - n \rme^{n b_0}/{|\zeta_n|}
\geq n'(1-\epsilon) - n\rme^{-n(b - b_0- 2\epsilon)} \geq n(1-2\epsilon) \,,
\end{equation}
provided $\epsilon$ is small enough so that $b - b_0 - 2\epsilon > 0$. This shows that $\zeta_n \in \cZ_n(b'')$ with 
\begin{equation}
\label{e:1.32}
b'' \in [b-\epsilon, b+\epsilon]
\quad, \qquad
n^{-1} \la \ol{\zeta}_n\,, x \ra \geq 1 - \epsilon \,,
\end{equation}
for arbitrarily small $\epsilon > 0$ and then $n$ large enough. This shows that the second statement of the proposition also holds when $b_1 -b > \log 2$ as $\gamma(b) = 1$ in this case, and finishes the proof.
\end{proof}

\section*{Acknowledgments}
The work of O.L. was supported in part by the European Union's - Seventh Framework Program (FP7/2007-2013) under grant agreement no 276923 -- M-MOTIPROX. E.T. was supported by a Technion MSc. fellowship.

\bibliographystyle{abbrv}
\bibliography{BRW_DistLD}
\end{document}